\documentclass[preprint,11pt
]{elsarticle}

\usepackage{amssymb}

\usepackage{amsmath,amssymb,amsfonts}%
\usepackage{amsthm}%
\usepackage{mathrsfs}%
\usepackage[title]{appendix}%
\usepackage{xcolor}%
\usepackage{textcomp}%
\usepackage{manyfoot}%
\usepackage{booktabs}%
\usepackage{algorithm}%
\usepackage{algorithmicx}%
\usepackage{algpseudocode}%
\usepackage{listings}%
\usepackage{tikz}
\newtheorem{thm}{Theorem}[section]
 \newtheorem{cor}[thm]{Corollary}
 \newtheorem{lem}[thm]{Lemma}

 \newtheorem{defn}[thm]{Definition}
 \newtheorem{rem}[thm]{Remark}
 \newtheorem{ex}[thm]{Example}

 \newcommand{\T}{\mathrm T}
\newcommand{\ip}[2]{\langle {#1}, {#2} \rangle}

\newcommand{\ra}[1]{\mathrm{rank}({#1})}  
 
\journal{...}

\numberwithin{equation}{section}

\begin{document}

\begin{frontmatter}



\title{\bf Characterizations of structured Bohemian matrices and their Inner and Outer Bohemian Inverses
}

\address[1]{Department of Mathematical and Computational Sciences, National Institute of Technology Karnataka,
Surathkal, Mangalore, 575025, Karnataka, India}

\address[2]{Faculty of Sciences and Mathematics, University of Ni\v s, Vi\v segradska 33, 18000 Ni\v s, Serbia}

\author[1]{Geeta Chowdhry}
 \ead{geetac.207ma003@nitk.edu.in}
\author[2]{Predrag S. Stanimirovi\' c
}
\ead{pecko@pmf.ni.ac.rs}
\author[1]{Falguni Roy}
\ead{royfalguni@nitk.edu.in}

\begin{abstract}

In this paper, we systematically define and characterize various classes of Bohemian matrices with respect to the population $\mathbb{P}=\{0, \pm 1\}$, focusing on their inner and outer Bohemian inverses. 
The classes under consideration include rank-one Bohemian matrices, as well as higher-rank Bohemian matrix classes, specifically Classes I, II, and III.  
For rank-one Bohemian matrices, a complete description of the outer Bohemian inverse sets is provided along with their cardinalities. 
Additionally, new insights into inner Bohemian inverses and their cardinalities are given. 
Characterizations of the complete set of inner inverses for the Class III matrices and full-row rank Class II matrices are obtained. 
Furthermore, we study the sets of rank-one outer inverses for Classes I and II, and examine the sets of rank $r$ outer inverses for full-row rank Class II matrices of rank $r$. 
Moreover, the set of outer inverses is completely characterized for rank-two full-row rank Class III matrices. 
In particular, we provide an explicit formula for the cardinality of the set of outer Bohemian inverses for rank-two full-row rank Class I matrices.

\end{abstract}




\end{frontmatter}


\section{Introduction}\label{SecIntro}

A family of Bohemian matrices refers to a collection of matrices whose entries are independently sampled from a finite set, typically consisting of integers with a defined upper limit (known as height). 
The term Bohemian matrix serves as a mnemonic for ``BOunded HEight Matrix of INtegers'' \cite{chan2022inner}. 
More precisely, a matrix $A\in \mathbb{K}^{m \times n} $ is termed as a {\it Bohemian matrix} if all of its entries belong to a specified set \(\mathbb{P}\), known as the population, which is typically discrete and finite. 
Although this class of matrices has been studied for a long time, it has not always been referred to by the name ``Bohemian matrix". 

Let $\mathbb{K}$ represent either the field of real numbers $\mathbb{R}$ or the field of complex numbers $\mathbb{C}$. 
The collection of all $m \times n$ matrices with entries from $\mathbb{K}$ is denoted by $\mathbb{K}^{m \times n}$. 
The set of matrices with rank $r$ in $ \mathbb{K}^{m \times n}$ is denoted by $ \mathbb{K}_r^{m \times n} $.
An $m \times n$ matrix $A$ will be often denoted by $A_{mn}$. 
The rank of $A\in \mathbb{K}^{m \times n}$ is denoted as $\ra{A}$, and the sum of all entries of $A$ is denoted as $\Xi(A)$. 
The transpose and conjugate transpose of $A$ are denoted by $A^\T$ and $A^*$, respectively. 
For vectors $x,y \in \mathbb{K}^{m \times 1}$, the notation $\ip{x}{y}$ denotes the Euclidean inner product $\ip{x}{y} :=x^* y$ in complex domain and $\ip{x}{y} :=x^\T y$ in the real domain.

The set of all $m$ by $n$ Bohemian matrices over a population $\mathbb{P}$ will be denoted as $\mathbb{P}^{m\times n}$. 
Some examples of Bohemian matrices include Mandelbrot matrices \cite{chan2019algebraic}, Bernoulli matrices \cite{tikhomirov2020singularity}, Metzler matrices \cite{briat2017sign}, Hadamard matrices \cite{horadam2012hadamard}, etc.
Bohemian matrices frequently appear in the analysis of linear systems, including solving such systems and computing their eigenvalues.  
Although low-dimensional Bohemian matrices are relatively simple, many open problems remain open in this topic.  
For instance, the exact number of singular $6 \times 6$ matrices with entries restricted to the set $\{ 0, \pm 1\}$ is still unknown \cite{thornton2019algorithms}.
Furthermore, if $A \in \mathbb{P}^{n \times n}$ is a Bohemian matrix, its inverse $A^{-1}$ does not necessarily be Bohemian with respect to the same population $\mathbb{P}$ \cite{martinez20210}. 
Matrices whose inverses remain Bohemian with respect to the same population are referred to as rhapsodic \cite{chan2022inner}.  
This naturally leads to the problem of identifying matrices that possess the rhapsodic property.  
It is evident that if the population $\mathbb{P}$ is a subfield of the field $\mathbb{K}$, then every nonsingular $n \times n$ matrix with entries in $\mathbb{P}$ is rhapsodic. 
Throughout this paper, unless otherwise specified, it will be assumed that the population is $\mathbb{P}=\{0, \pm 1\}$.

A natural research problem is to investigate the rhapsodic behaviour of generalized inverses.
Significant contributions in this direction can be found in \cite{chu2011magic,chan2022inner} and \cite{chowdhry2025characterizations}. 
For a matrix $A \in \mathbb{K}^{m \times n}$, a matrix $X\in \mathbb{K}^{n\times m}$ is known as the Moore-Penrose inverse of $A$ if it satisfies the following four Penrose equations \cite{penrose1955generalized}
$$ (1)\ AXA=A, \ (2) \ XAX=X, \ (3) \ AX=(AX)^*, \ (4) \ XA=(XA)^*.$$
Inner and outer generalized inverses are important in the literature.
A matrix $X \in \mathbb{K}^{n \times m}$ is referred to as an inner inverse or $\{1\}$-inverse of $A \in \mathbb{K}^{m \times n}$ if the matrix equation $AXA=A$ holds. 
If $X$ satisfies $XAX=X$, it is known as an outer inverse or $\{2\}$-inverse of $A$. 
The set of all $\{1\}$-inverses of $A$ is denoted as $A\{1\}$, and the set of all $\{2\}$-inverses of $A$ is denoted as $A\{2\}$. 
Additionally, the set of outer inverses with prescribed rank $s$ is denoted as $A\{2\}_s$. 
In general, the set of all $\{i,j,k, \ldots \}$-inverses (satisfying equations $(i),(j),(k), \ldots $) of $A$ is denoted as $A\{i,j,k,\ldots\}$.

Recent studies have focused on the inner inverses of Bohemian matrices, as detailed in \cite{chan2022inner}.
These studies provide new insights into the structure of inner inverses and inner Bohemian inverses of full and well-settled structured Bohemian matrices. 
For a matrix $A \in \mathbb{P}^{m \times n}$, an inner Bohemian inverse of $A$ is defined as a matrix $X \in \mathbb{K}^{n \times m}$ satisfying  
$
X \in A\{1\} \cap \mathbb{P}^{n \times m}.
$
This condition ensures $X \in A\{1\}$, while entries of $X$ remain within the specified population $\mathbb{P}$. 
Analogous to the definition of inner Bohemian inverse, $X \in \mathbb{K}^{n \times m}$ is said to be an outer Bohemian inverse of $A \in \mathbb{P}^{m \times n}$ if $X \in A\{2\} \cap \mathbb{P}^{n \times m}$. 
In general, $X \in \mathbb{K}^{n \times m}$ is referred to as $\{i,j,\ldots\}$-Bohemian inverse of $A \in \mathbb{P}^{m \times n}$ if $X \in A\{i,j,\ldots\} \cap \mathbb{P}^{n \times m}$ and
$A_{\mathbb{P}}\{i,j,\ldots\}=\{X|\ X \in A\{i,j,\ldots\} \cap \mathbb{P}^{n \times m}\}$. 
Moreover, several open problems are mentioned in \cite{chan2022inner}. 
In \cite{chowdhry2025characterizations}, the open problem 2 has been addressed. 
Furthermore, \cite{chowdhry2025characterizations}, provides characterizations of rank-one Bohemian matrices and new insights into the structure of $\{1,3\}$ and $\{1,4\}$-Bohemian inverses are provided for rank-one Bohemian matrices and well-settled matrices. 
Additionally, both the works \cite{chan2022inner,chowdhry2025characterizations}, determine the cardinalities of the sets of the concerned Bohemian inverses of these structured matrices.

Our goal is to characterize certain extended classes of Bohemian matrices and to examine representations of both outer and inner inverses of these matrices. 
Additionally, the cardinality of the sets of inner and outer Bohemian inverses of specific structured matrices is studied.

Significant research has been conducted independently on generalized inverses and Bohemian matrices. 
Some notable results on Bohemian matrices can be found in several studies  
\cite{taussky1960matrices,taussky1961some,thornton2019algorithms,barrett2016symmetric,martinez20210,guyker2007magic}. 
Additional findings related to Bohemian matrices are presented in the references \cite{CHAN2020,SendraBoh1,SendraBoh2,thornton2019algorithms,Boh3}.
These results encompass the examination of matrices with integer entries, the rhapsodic behavior of inverses with weaker constraints when the inverse is similar to a Bohemian matrix, and the inversion of equimodular matrices. 
Research on magic squares with magic inverses \cite{guyker2007magic} and magic pseudoinverses \cite{chu2011magic} has also been conducted.
Furthermore, the inverse of a nonsingular semi-magic square and the Moore-Penrose inverse of a semi-magic square are semi-magic squares \cite{Schmidt01072001}.

The literature on generalized inverses is extensive. 
Some notable works on representations of the outer and inner inverses of matrices include 
\cite{wei1998characterization,sheng2007full,stanimirovic2011full,stanimirovic2017conditions,stanimirovic2022representations}.
These studies provide representations and characterizations of outer and inner inverses, as well as their subsets, and explore geometric properties of the generalized inverses.
However, many existing representations of inner and outer inverses of matrices are ineffective for determining inner and outer Bohemian inverses. 
These representations fail to clarify whether the sets of inner or outer Bohemian inverses are empty or non-empty. 
Therefore, our objective is to provide explicit representations of both the outer and inner inverses of Bohemian matrices.
We also aim to determine the cardinalities of specific subsets of outer and inner Bohemian inverses.

The presentation of this paper is outlined as follows.
Motivation and preliminary are presented in Section \ref{secMotiv}.
In Section \ref{sec3}, we characterize rank-one Bohemian matrices and generalized well-settled matrices for the population $\mathbb{P}=\{0, \pm 1\}$. 
We also provide simplified findings regarding Inner Bohemian inverses in \cite{chan2022inner}, along with new insights into the cardinalities of inner Bohemian inverses across different classes of structured Bohemian matrices.
In section \ref{sec 4}, we define and characterize several extended classes of Bohemian matrices and study their inner Bohemian inverses.
Section \ref{sec5} focuses on the study of Bohemian inverses for various classes, as well as the cardinalities of sets of outer Bohemian inverses for specific structured matrices.

\section{Motivation and preliminaries} \label{secMotiv}

Let $\mathcal{S} \subseteq \mathbb{K}^{m \times n}$ be a set of matrices. 
If $\mathcal{S}$ represents a finite set, then the notation $\# \mathcal{S}$ denotes the cardinality of the set $\mathcal{S}$, which refers to the number of matrices it contains. 
The expression $U \mathcal{S} V$ denotes the set
$U \mathcal{S} V := \{ USV \mid S \in \mathcal{S} \}$, where $U$, $V$ are matrices of compatible sizes. 
Similarly, we use $U \mathcal{S} := \{ US \mid S \in \mathcal{S} \}$ and $\mathcal{S} V := \{ SV \mid S \in \mathcal{S} \}$ to denote left and right multiplication, respectively. 
For a scalar $\lambda \in \mathbb{K}$, the scaled set $\mathcal{S}$ is written as
$\lambda \mathcal{S} := \{ \lambda S \mid S \in \mathcal{S} \},$
and the set of conjugate transposes is given by
$\mathcal{S}^* := \{ S^* \mid S \in \mathcal{S} \},$
where $S^*$ denotes the conjugate transpose of $S$.

Next, we will restate some properties of outer inverses.
\begin{lem} \label{1.1}
    Let $A \in \mathbb{K}^{m \times n}$, let $U \in \mathbb{K}^{m \times m}$ and $V \in \mathbb{K}^{n \times n}$ be unitary matrices and let $0 \neq \lambda \in \mathbb{K}$. 
    It holds that
    \begin{enumerate}    
        \item $(\lambda A)\{i\}=\frac{1}{\lambda}(A\{i\})$,
        \item $(UAV)\{i\}=V^*(A\{i\})U^*$,
        \item 
        $(A\{i\})^*=A^*\{i\}$,
    \end{enumerate}
where $i \in \{1,2\}$.
\end{lem}
Corollary \ref{cardinaliy} follows from Lemma \ref{1.1}.
\begin{cor} \label{cardinaliy}
    Let $\mathbb{P}=\{0, \pm 1\}$ and $A \in \mathbb{P}^{m \times n}$, let $U \in \mathbb{K}^{m \times m}$ and $V \in \mathbb{K}^{n \times n}$ be unitary matrices and let $0 \neq \lambda \in \mathbb{K}$. 
   Then it holds that
    \begin{enumerate}
      
        \item $\#(\lambda A)_{\mathbb{P}}\{i\}=\#(A_{\mathbb{P}}\{i\})$= $\#(UAV)_{\mathbb{P}}\{i\}$,
        \item 
        $\#(A_{\mathbb{P}}\{i\})^*=\#A^*_{\mathbb{P}}\{i\}$,
    \end{enumerate}
where $i \in \{1,2\}$.
\end{cor}
\begin{proof}
    For $X \in \mathbb{P}^{n \times m}$, the maps $X \rightarrow  \frac{1}{\lambda} X$, $X \rightarrow V^*XU^*$ and $X \rightarrow X^*$ are bijective. Hence the proof.
\end{proof}

\begin{defn} {\rm \cite{horn2012matrix}} Two $m \times n$ matrices $A$ and $B$ are known as unitary equivalent if there exist unitary matrices $U$ and $V$ such that $B=UAV$. 
Similarly, $A$ and $B$ are known as permutation equivalent if $B=PAQ$ holds for some permutation matrices $P$ and $Q$.
\end{defn}

Let $A \in \mathbb{K}^{m \times n}$ and $B$ be any matrix unitary equivalent to $A$, satisfying $B=UAV$ for some unitary matrices $U$ and $V$. 
According to Corollary \ref{cardinaliy}, it is clear that the sets of generalized inverses $A\{i\}$ and $B\{i\}$ for 
$i \in \{1,2\}$ are different; however, the cardinalities of these sets remain the same. 
Therefore, for unitary equivalent matrices, the cardinalities of the sets of the outer Bohemian inverses and the inner Bohemian inverses are preserved, meaning $\#A_{\mathbb{P}}\{i\}$= $\#B_{\mathbb{P}}\{i\}$ for $i \in \{1,2\}$. 
This fact will be used throughout the paper. 
For two unitary equivalent matrices $A$ and $B$, the sets $A\{i\}$ and $B\{i\}$ may differ; however, knowing one set allows us to easily derive the other. 
Moreover, the Bohemian unitary equivalence satisfies the condition $B_{\mathbb{P}}\{i\}=(UAV )_{\mathbb{P}}\{i\} = V^\T(A_{\mathbb{P}}\{i\})U^\T \subset \mathbb{P}^{n \times m}$ for the special case when $U$ and $V$ are invertible Bohemian diagonal matrices.
The same conclusions hold for the special case of permutation equivalence. 
In the upcoming sections, we will utilize these facts of unitary and permutation equivalence to explore the generalized inverses of a simpler class of matrices. 
This analysis will assist us in determining the generalized inverses of a more complex class of matrices.
The following result provides a representation of the outer inverses with the prescribed rank for a matrix containing some zero columns.

The block matrix $A=(B \ |\ C) \in \mathbb{K}^{m \times {(n_1+n_2)}}$ means $A$ is a partitioned matrix such that $B  \in \mathbb{K}^{m \times {n_1}}$ and $C  \in \mathbb{K}^{m \times {n_2}}$.
Similarly, $A=\begin{pmatrix}    B \\ C\end{pmatrix} \in \mathbb{K}^{(m_1+m_2) \times {n}}$ 
means the matrix $A$ partitioned into the blocks $B\in \mathbb{K}^{m_1 \times {n}}$ and $C\in \mathbb{K}^{m_2 \times {n}}$.

\begin{lem} \label{1.2}
Outer inverses of $A=(B_{mn_1}\ |\ 0_{mn_2}) \in  \mathbb{K}^{m \times (n_1+n_2)}$ of rank $s\leq \ra{A}$ are defined by the set
\begin{equation*}
\aligned
    A\{2\}_s&=\left\{\begin{pmatrix}
            X_1 \\
            X_2
        \end{pmatrix} :\ X_1 \in B\{2\}_s\in  \mathbb{K}^{n_1 \times m}, \  X_2BX_1=X_2\in  \mathbb{K}^{n_2 \times m}\right\}\\
        &=\left\{\begin{pmatrix}            X_1 \\            X_2        \end{pmatrix} :\ X_1 =P Q \in  \mathbb{K}^{n_1 \times m}, P\in  \mathbb{K}^{m \times s},Q \in  \mathbb{K}^{s \times n_1}, QBP=I_s, \  X_2B PQ=X_2\in  \mathbb{K}^{n_2 \times m}\right\}.
        \endaligned
\end{equation*}
\end{lem}

\begin{proof}
The first part of the proof is derived from
$$
\begin{pmatrix}
            X_1 \\             X_2
        \end{pmatrix} (B \ |\ 0) \begin{pmatrix}            X_1 \\            X_2        \end{pmatrix}=
\begin{pmatrix}
            X_1BX_1 \\            X_2B X_1        \end{pmatrix},
$$
and the second part comes from a known representation of outer inverses with prescribed rank \cite[P. 57, Theorem 5]{benbook}.
\end{proof}

We will now introduce some additional notations that will be important for future discussions.
We denote by $(0_{mn}), (1_{mn}), \text{ and } ({-1}_{mn})$ the zero matrix, the matrix with all entries equal to $1$ and the matrix with all entries equal to $-1$, respectively. 
Furthermore, the notation $(\pm 1_{mn_1} \ | \ \mp 1_{mn_2})$ will refer simultaneously the matrices $(1_{mn_1} \ | \, -1_{mn_2})$ and $(-1_{mn_1} \ | \ 1_{mn_2})$.

First, we restate the notion of full and well-settled matrices defined in \cite{chan2022inner}.
\begin{defn}\label{def1} {\rm \cite{chan2022inner}}
    A matrix $A \in \mathbb{K}^{m \times n}$ is said to be a full matrix if it takes any of the following forms:
    \begin{enumerate}
        \item Type I full matrix: $(\pm 1_{mn})$.
        \item Type II full matrix: $(\pm 1_{mn_1} | \mp 1_{mn_2})$.
        \item Type III full matrix: $(\pm 1_{mn_1} \ |\ 0_{mn_2})$.
        \item Type IV full matrix: $(\pm 1_{mn_1}\  |\ \mp 1_{mn_2}\ | \ 0_{mn_3})$.
    \end{enumerate}
\end{defn}

Lemma \ref{multiply full matrix} explores the impact of left and right multiplication by specific full matrices.

\begin{lem} {\rm \cite{chan2022inner}} \label{multiply full matrix}  Let $X = \begin{pmatrix} X_1 \\ X_2 \end{pmatrix} \in \mathbb{K}^{(n_1+n_2) \times m}$, where $n_1+n_2=n$. 
Then
    \begin{enumerate}
  \item $(1_{mn}) \, X \, (1_{rs}) = \Xi(X) \, (1_{m s})$.
  \item $(1_{m  n}) \, X (1_{r s_1} \ | \ 1_{r s_2})   = \Xi(X) (  1_{m  s_1}\ | \  1_{m  s_2})$.
  \item
  $( 1_{m  n_1} \ | -1_{m n_2}  X ( 1_{r  s_1} \ |\,  -1_{r  s_2})
  = ( \Xi(X_1) - \Xi(X_2)) (1_{m  s_1} \ |\,  -1_{m  s_2}).$
  \item $(1_{m  n_1} \ | \,  -1_{m  n_2})  X \, (1_{r s}) = \left( \Xi(X_1) - \Xi(X_2) \right) (1_{m  s}).$
\end{enumerate}
\end{lem}

\begin{defn} {\rm \cite{chan2022inner}} 
A matrix $A \in \mathbb{K}^{m \times n}$ is said to be a well-settled matrix if, upon multiplication by appropriate permutation matrices $P$ and $Q$, the product $PAQ$ assumes a general form
\begin{equation} \label{ws def}
    PAQ= \begin{pmatrix}
   A_1 & \hdots & 0\\
        \vdots & \ddots & \vdots\\
        0 & \hdots& A_s \end{pmatrix},
\end{equation}
where each $A_i \in \mathbb{K}^{m_i \times n_i}$, $i\in \{1,\ldots s\}$ is a full matrix for positive integers $m_i$ and $n_i$ satisfying $m_1+\cdots +m_s=m$ and $n_1+\cdots +n_s=n$. 
Also, $A$ is called pure well-settled if all blocks $A_i$ are of the same type; otherwise, it is categorized as mixed well-settled.
\end{defn}

Note that the general form of well-settled matrices is as specified in equation \eqref{ws def}. 
In this context, each $A_i$ is restricted to be a full matrix. 
Since the full matrices are rank-one matrices, in Definition \ref{def gws} we introduce a broader class of Bohemian matrices that includes the class of well-settled matrices.

 \begin{defn} \label{def gws}
  A matrix $A \in \mathbb{K}^{m \times n}$ is considered a generalized well-settled matrix if it takes the specific form 
     \begin{equation} \label{gws eq}
    PAQ= \begin{pmatrix}
   A_1 & \hdots & 0\\
        \vdots & \ddots & \vdots\\
        0 & \hdots& A_s \end{pmatrix},
\end{equation}
where $P$ and $Q$ are suitable permutation matrices and all blocks $A_i\in \mathbb{K}^{m_i \times n_i}$, $i=1,\ldots s$ are rank-one Bohemian matrices w.r.t. the population $\mathbb{P}=\{0, \pm 1\}$.
 \end{defn}

\begin{ex} \label{exgws} The matrix $A=\begin{pmatrix}
        1 & -1 &0 &0 \\
        -1 & 1 & 0 & 0 \\
        0 & 0 & 1 & -1
    \end{pmatrix}$
is not a well-settled matrix, but it is a generalized well-settled matrix.
\end{ex}

 \begin{rem}
From Definition \ref{def gws} and Example \ref{exgws}, it is evident that the class of well-settled matrices is contained in the class of generalized well-settled matrices, but the converse is not true.
 \end{rem}

Highlights of derived results are as follows.
\begin{enumerate}
\item Various classes of Bohemian matrices are defined to study representations of inner and outer Bohemian inverses. 

    \item Defined classes include the rank-one Bohemian matrices; for higher rank matrices, we introduced Classes I, II and III.
    
    \item  Rank-two Class III matrices are divided into four types of structured matrices, and inner inverses of Class III matrices and full-row rank Class II matrices are studied.

   \item  Outer inverses of the prescribed rank one are studied for Class I and Class II. 

   Particularly, outer inverses for rank-two Class III matrices are completely characterized. 
   In conclusion, we have worked on the open problem 3 outlined in \cite{chan2022inner} focusing on the outer inverses of certain classes of Bohemian matrices.
   
   \item The cardinality of outer Bohemian inverses is determined for rank-one Bohemian matrices and particular well-settled matrices. 
   
   \item New insights are provided on the cardinalities of the inner Bohemian inverses.

\end{enumerate}

\section{Inner inverses of rank-one Bohemian and generalized well-settled matrices} \label{sec3}

In this section, we characterize rank-one Bohemian matrices and clarify specific findings regarding inner Bohemian inverses as presented in \cite{chan2022inner}.
We also provide new insights into the cardinalities of inner Bohemian inverses across different classes of structured Bohemian matrices. 
Furthermore, characterizations of generalized well-settled matrices are provided.
Our analysis concludes that determining inner Bohemian inverses of well-settled matrices is sufficient for finding these inverses for generalized well-settled matrices. 
 Furthermore, new results on the cardinalities of the inner inverses of well-settled matrices are provided.

\subsection{Rank-one Bohemian matrices} \label{Sec 3.1}

We recall the following result from \cite{chowdhry2025characterizations}, which offers a representation theorem for the rank-one Bohemian matrices and plays a vital role for their characterization.

\begin{thm} \label{th 2.1} {\rm \cite{chowdhry2025characterizations}}
    Let $\mathbb{P}=\{0, \pm 1\}$. 
    Any rank-one matrix $A \in \mathbb{P}^{m \times n}$ can be represented as $A=P_1DA'P_2$, where $P_1, P_2$ are appropriate permutation matrices, $D \in \mathbb{P}^{m \times m}$ is an invertible diagonal matrix.
In the case $A$ has no zero rows, $A' \in \mathbb{P}^{m \times n}$ is a full matrix; otherwise if $A$ has $k$ zero rows then
 $A'=\begin{pmatrix}
        B \\0
    \end{pmatrix}$, where $B \in \mathbb{P}^{(m-k) \times n}$ is a full matrix.
\end{thm}

The following result provides an improved representation of rank-one Bohemian matrices and assists in finding generalized inverses of specific Bohemian matrices.

\begin{thm} \label{th 2.1 new}
    Let $\mathbb{P}=\{0,\pm 1\}$. Any rank-one $A \in \mathbb{P}^{m \times n}$ can be represented as $A=P_1D_1A'D_2P_2$, where $P_1, P_2$ are appropriate permutation matrices, $D_1 \in \mathbb{P}^{m \times m}$ and $D_2 \in \mathbb{P}^{n \times n}$ are invertible diagonal matrices.
In the case $A$ has no zero rows, $A' \in \mathbb{P}^{m \times n}$ is a full matrix of type I or type III; otherwise if $A$ has $k$ zero rows then
 $A'=\begin{pmatrix}
        B \\0
    \end{pmatrix}$, where $B \in \mathbb{P}^{(m-k) \times n}$ is a full matrix of type I or type III.
\end{thm}
\begin{proof}

By Theorem \ref{th 2.1}, $A$ is represented as $A=P_1DA'P_2$, where in the case $A$ has no zero rows, $A' \in \mathbb{P}^{m \times n}$ is a full matrix; otherwise if $A$ has $k$ zero rows then $A'=\begin{pmatrix}
        B \\0
    \end{pmatrix}$, where $B \in \mathbb{P}^{(m-k) \times n}$ is a full matrix.

    The full matrices of type II and type IV can be rewritten in terms of matrices of type I and type III as follows:
    \begin{equation} \label{eq thm 3.2}
    \left(\pm 1_{mn_1} \ | \ \mp 1_{mn_2}\right)=\left(1_{m(n_1+n_2)}\right) \begin{pmatrix}
        \pm I_{n_1} & 0 \\ 0 & \mp I_{n_2} 
    \end{pmatrix}
    \end{equation}
    and 
    \begin{equation} \label{eq thm 3.2 2}
    \left(\pm 1_{mn_1} \ | \ \mp 1_{mn_2} \ | \ 0_{mn_3} \right) = \left( 1_{m(n_1+n_2)} \ | \ 0_{mn_3}\right) \begin{pmatrix}
        \pm I_{n_1} & 0 & 0 \\ 0 & \mp I_{n_2}  & 0 \\ 0&0& I_{n_3}
    \end{pmatrix}. 
     \end{equation}
    Further, by applying equations \eqref{eq thm 3.2} and \eqref{eq thm 3.2 2} for type II and type IV matrices, the proof is concluded.
\end{proof}

\begin{ex}
    Consider $A=\begin{pmatrix}
        1 & -1 & 0 & 1 \\
        -1 & 1 & 0& -1 \\
        1 & -1 & 0 & 1 \\
        0 & 0 & 0 & 0
    \end{pmatrix}$. With the help of permutation matrices and diagonal matrices, $A$ can be decomposed as in Theorem \ref{th 2.1 new} as follows:
    \begin{equation*}
        \begin{split}
            A 
    &= \begin{pmatrix}
        1 & 0 & 0 & 0 \\
        0 & -1 & 0& 0 \\
        0 & 0 & 1 & 0 \\
        0 & 0 & 0 & 1
    \end{pmatrix} \begin{pmatrix}
        1 & 1 & 1 & 0 \\
        1 & 1 & 1& 0 \\
        1 & 1 & 1 & 0 \\
        0 & 0 & 0 & 0
    \end{pmatrix} \begin{pmatrix}
        1 & 0 & 0 & 0 \\
        0 & 1 & 0& 0 \\
        0 & 0 & -1 & 0 \\
        0 & 0 & 0 & 1
    \end{pmatrix} \begin{pmatrix}
        1 & 0 & 0 & 0 \\
        0 & 0 & 0& 1 \\
        0 & 1 & 0 & 0 \\
        0 & 0 & 1 & 0
    \end{pmatrix}=P_1D_1A^{'}D_2P_2,
        \end{split}
    \end{equation*}
    where $P_1$ is the identity matrix in this case and $A^{'}$ is of the form $\begin{pmatrix}
        B \\ 0
    \end{pmatrix}$ where $B=(1_{33} \ | \ 0_{31})$ is a type III matrix. 
\end{ex}

\begin{rem} \label{rem rk 1}
    In conclusion, any rank-one Bohemian matrix with respect to the population $\mathbb{P}=\{0, \pm 1\}$ can be written as $A=UBV$, where $U=P_1D_1$ and $V=D_2P_2$ are unitary matrices of appropriate order. 
If $A$ has no zero rows, then $B \in \mathbb{P}^{m \times n}$ is a full matrix of type I or type III; otherwise if $A$ has $k$ zero rows then
 $B=\begin{pmatrix}
        F \\0
    \end{pmatrix}$, where $F \in \mathbb{P}^{(m-k) \times n}$ is a full matrix of type I or type III. 
 Also, the construction of $U$ and $V$ ensures that $U$ and $V$ are Bohemian matrices over the same population. It also guarantees that multiplying any Bohemian matrix $A$ by such a $U$ or $V$ preserves the fact that the matrix $A$ remains Bohemian.

\end{rem}

\begin{rem} \label{remark for rankone}
According to Lemma \ref{1.1}, Theorem \ref{th 2.1 new}, and \cite[Lemma 2]{chan2022inner}, to find the $\{1\}$-inverses of rank-one matrices, it is sufficient to determine the $\{1\}$-inverses of the full matrices of type I. 
Similarly, based on Lemma \ref{1.1}, Lemma \ref{1.2}, and Theorem \ref{th 2.1 new}, to find the $\{2\}$-inverses of rank-one matrices, it is sufficient to find the $\{2\}$-inverses of the full matrices of type I and type III.
\end{rem}

The inner Bohemian inverses of full matrices of type I and type II have been studied in \cite{chan2022inner}. 
It has been established that the results for full matrices of type II can also be derived entirely from the findings related to full matrices of type I. 
The following result characterizes $\{1\}$-inverses of full matrices of type II, which is a consequence of Remark \ref{rem rk 1} and \cite[Theorem 1]{chan2022inner}.

\begin{thm} \label{type ii}
    If $A=( \pm 1_{mn_1} \ | \  \mp 1_{mn_2})$ satisfies $n_1,  n_2 \neq 0$ and $n_1+n_2=n$, then it follows
    \begin{equation*}
        A\{1\}=\left\{\begin{pmatrix}
        \pm I_{n_1} & 0 \\ 0 & \mp I_{n_2} 
    \end{pmatrix} \begin{pmatrix}
        1-\displaystyle \sum_{k=1}^{mn-1}x_k & x_1 & \cdots & x_{m-1}\\
        x_m & x_{m+1} & \cdots & x_{2m-1} \\
        \vdots & \vdots &  & \vdots\\
        x_{(n-1)m} & x_{(n-1)m+1} & \cdots & x_{nm-1}
    \end{pmatrix} : x_i \in \mathbb{K} \right\}.
    \end{equation*}
\end{thm}

\begin{rem}
The result obtained in Theorem \ref{type ii} provides an alternative formulation of {\rm \cite[Theorem 1(2)]{chan2022inner}}. 
In contrast to the method applied in {\rm\cite[Theorem 1(2)]{chan2022inner}}, which requires two sets of parameters, our result is derived using only a single set of parameters.
\end{rem}

In the following result, we present concise proof for the result obtained in \cite[Theorem 2(2)]{chan2022inner}, for the type II full matrix $(\pm 1_{mn_1} \ | \ \mp 1_{mn_2})$ with a different approach.
\begin{thm} \label{th 2.6}
    If $A=( \pm 1_{mn_1}  \ | \  \mp 1_{mn_2})$ satisfies $n_1,  n_2 \neq 0$ and $n_1+n_2=n$, then it follows
\begin{equation}\label{Equin1} 
A\{1\}={\pm}\left\{\begin{pmatrix}
            X_1 \\
            X_2
        \end{pmatrix} \in  \mathbb{K}^{(n_1+n_2)\times m}: \Xi(X_1)-\Xi(X_2)=1 \right\}.
        \end{equation}     

\end{thm}

\begin{proof} According to equation \eqref{eq thm 3.2} and Lemma \ref{1.1}(2), it follows
    \begin{equation} \label{eq 4}
         (\pm 1_{mn_1} \ | \ \mp 1_{mn_2})\{1\}= \begin{pmatrix}
        \pm I_{n_1} & 0 \\ 0 & \mp I_{n_2} 
    \end{pmatrix}(1_{mn})\{1\}.
     \end{equation} 
    Incorporating the result of $(1_{mn})\{1\}$ from \cite[Theorem 2(1)]{chan2022inner} in the right-hand side of the equation \eqref{eq 4}, it can be obtained
    \begin{equation*}
        \begin{split}
            (\pm 1_{mn_1} \ | \ \mp 1_{mn_2})\{1\} &= \left\{\begin{pmatrix}
        \pm I_{n_1} & 0 \\ 0 & \mp I_{n_2} 
    \end{pmatrix}X \in \mathbb{K}^{n\times m} : \Xi(X)=1 \right\}\\
    &=\left\{\pm \begin{pmatrix}      X_1 \\      -X_2         \end{pmatrix} \in  \mathbb{K}^{(n_1+n_2)\times m}: \Xi(X_1)+\Xi(X_2)=1 \right\},
        \end{split}
    \end{equation*}
which aligns with \eqref{Equin1}.
\end{proof}

\begin{cor} \label{cor 3.7}
    For $A=(1_{mn})$ and $B=(1_{mn_1} \ |\,  -1_{mn_2})$ such that $n_1+n_2=n$. If $\mathbb{P}=\{0,\pm 1 \}$, then $\#(A_\mathbb{P}\{1\})=\#(B_\mathbb{P}\{1\})$.
\end{cor}

\begin{proof}
    The proof follows from Corollary \ref{cardinaliy} and equation \eqref{eq 4}.
\end{proof}

Theorem \ref{rem 3.7} extends the results discussed earlier.
\begin{thm} \label{rem 3.7}
    For any two rank-one Bohemian matrices $A, B \in \mathbb{P}^{m \times n}$ that have the same number of zero rows and the same number of zero columns, it follows $\#(A_\mathbb{P}\{1\})=\#{(B_\mathbb{P}\{1\})}$, where $\mathbb{P}=\{0,\pm 1 \}$.
\end{thm}
\begin{proof}
    The proof follows from Theorem \ref{th 2.1 new}, Lemma \ref{1.1} and Corollary \ref{cardinaliy}.
\end{proof}

In \cite{chan2022inner}, the authors derived separate expressions for $\#(A_\mathbb{P}\{1\})$ and $\#(B_\mathbb{P}\{1\})$, where $A$ and $B$ are defined as in Corollary \ref{cor 3.7}. 
However, in Corollary \ref{cor 3.7}, we directly established the equality of these cardinalities. 
Consequently, Corollary \ref{cor 3.7} in conjunction with \cite[corollary 2]{chan2022inner} yields the following identity.
Additionally, it is important to note that ${n \choose r}$ refers to the binomial coefficient, which is assumed to be zero whenever it is not defined.

\begin{thm} The following equality holds, where $n,m,n_1,n_2$ are natural numbers and $n_1+n_2=n$.
    \begin{equation*}
        \sum_{s_1=0}^{\lfloor \frac{nm-1}{2} \rfloor} {nm \choose s_1}{{nm-s_1} \choose {s_1+1}}=\sum_{r_2=0}^{n_2m} \sum_{s_2=0}^{n_2m} \sum_{s_1=0}^{\lfloor \frac{(n_1+n_2)m-1}{2}-r_2 \rfloor}  {n_1m \choose s_1} {n_2m \choose s_2}{{n_2m-s_2} \choose {r_2}}{{n_1m-s_1} \choose {r_2-s_2+s_1+1}},
    \end{equation*}
    where $\lfloor x \rfloor$ represents the greatest integer less than or equal to $x$.
\end{thm}

\subsection{Generalized well-settled matrices} \label{sec 3.2}

 Similar to the case with rank-one Bohemian matrices, we will show that identifying inner inverses of the $(\pm1_{mn})$-pure well-settled matrices is sufficient to determine the inner inverses of any well-settled matrix. 
 Furthermore, by Lemma \ref{1.1}(2), without loss of generality, the matrices $P$ and $Q$ are assumed to be identity matrices of the appropriate order. 
To begin, we will focus on the $(\pm1_{mn_1}\ |\ \mp1_{mn_2})$-pure well-settled matrices. 
Based on Lemma \ref{1.1}(1), it is suffices to examine $(1_{mn_1}\ |\, -1_{mn_2})$-pure well-settled matrices.

Let $A=\begin{pmatrix}
        A_1 & \hdots & 0\\
        \vdots & \ddots & \vdots\\
         0 & \hdots& A_s \end{pmatrix}$, $A_i= (1_{m_in_{i1}}\ |\, -1_{m_in_{i2}}) \in \mathbb{P}^{m_i \times n_i} $, where $i\in \{1,2,\ldots s\}$. 
Using the equation \eqref{eq thm 3.2} and $A_i=(1_{m_in_{i1}}\ |\, -1_{m_in_{i2}})=(1_{m_i(n_{i1}+n_{i2})}) \begin{pmatrix}
         I_{n_{i1}} & 0 \\ 0 & - I_{n_{i2}} 
    \end{pmatrix}$ for each block $A_i$, the matrix $A$ can be rewritten as follows:
    \begin{equation*}
        \begin{split}
            A =\begin{pmatrix}
        A_1 & \hdots & 0\\
        \vdots & \ddots & \vdots\\
         0 & \hdots& A_s \end{pmatrix}
         =\begin{pmatrix}
        1_{m_1n_1} & \hdots & 0\\
        \vdots & \ddots & \vdots\\
         0 & \hdots& 1_{m_sn_s} \end{pmatrix} \begin{pmatrix}
             I_{n_{11}} & 0 & \hdots & 0 &0\\
             0 & -I_{n_{12}}  & \hdots & 0 & 0 \\
        \vdots & \vdots & \ddots & \vdots & \vdots\\
        0 & 0 & \hdots & I_{n_{s1}} & 0\\
         0 & 0&  \hdots & 0 & -I_{n_{s2}}
         \end{pmatrix}.
        \end{split}
    \end{equation*}
    Hence, Lemma \ref{1.1}(2) implies
    \begin{equation*}
        A\{1\}=\begin{pmatrix}
             I_{n_{11}} & 0 & \hdots & 0 &0\\
             0 & -I_{n_{12}}  & \hdots & 0 & 0 \\
        \vdots & \vdots & \ddots & \vdots & \vdots\\
        0 & 0 & \hdots & I_{n_{s1}} & 0\\
         0 & 0&  \hdots & 0 & -I_{n_{s2}}
         \end{pmatrix} \begin{pmatrix}
        1_{m_1n_1} & \hdots & 0\\
        \vdots & \ddots & \vdots\\
         0 & \hdots& 1_{m_sn_s} \end{pmatrix} \{1\}.
    \end{equation*}
The results concerning the inner inverses of the $(\pm1_{mn})$-pure well-settled matrices, which are thoroughly studied in \cite{chan2022inner}, can be utilized to characterize the set $A\{1\}$. 
Similarly, the findings related to mixed well-settled matrices can be derived solely from the outcomes concerning the $(\pm1_{mn})$-pure well-settled matrices.
This is feasible because, for each block $A_i$, the relation \eqref{eq thm 3.2} can be applied whenever $A_i$ is a full matrix of type II. 

\smallskip
In conclusion, determining the inner inverses of the $(\pm1_{mn})$-pure well-settled matrices is sufficient for establishing the inner inverses of any well-settled matrix.
This is the essence of Theorem \ref{car 3.9}.

\begin{thm} \label{car 3.9}
Let $A=\begin{pmatrix}
        A_1 & \hdots & 0\\
        \vdots & \ddots & \vdots\\
         0 & \hdots& A_s \end{pmatrix}$ with blocks $A_i=\epsilon_i (1_{m_in_i}) $, 
$B=\begin{pmatrix}        B_1 & \hdots & 0\\      \vdots & \ddots & \vdots\\
         0 & \hdots& B_s \end{pmatrix}$ with $B_i=\epsilon_i (1_{m_in_{i1}}\ |\, -1_{m_in_{i2}})$, 
         and $C=\begin{pmatrix}        C_1 & \hdots & 0\\        \vdots & \ddots & \vdots\\         0 & \hdots & C_s \end{pmatrix}$
        with possible choices for matrices $C_i$ given by
         $C_i=\epsilon_i (1_{m_in_i})$ and $C_i=\epsilon_i (1_{m_in_{i1}}\ |\, -1_{m_in_{i2}})$, 
        where $A_i, \ B_i, \ C_i \in \mathbb{K}^{m_i \times n_i}$ and $\epsilon_i \in \{-1,1\}$ for all $i \in \{1,\ldots ,s\}$. 
Then 
$$\#(A_\mathbb{P}\{1\})=\#(B_\mathbb{P}\{1\})=\#(C_\mathbb{P}\{1\}).$$
\end{thm}

\begin{rem}
According to Theorem \ref{car 3.9}, three formulas for cardinalities, as obtained in {\rm \cite[Corollary 3]{chan2022inner}}, are shown to be equal. 
The cardinality expressions can be found in {\rm \cite{chan2022inner}}.
\end{rem}

In Theorem \ref{gws ch1}, we characterize the generalized well-settled matrices.

\begin{thm} \label{gws ch1}
Let  $\mathbb{P} =\{0, \pm 1\}$. 
The matrix $A \in \mathbb{K}^{m \times n}_r$ is a generalized well-settled matrix if and only if up to the permutation equivalence, $A$ can be written in the form $A=\displaystyle \sum_{i=1}^{r} {u_iv_i^\T}$, where $u_i=(u_{i1}, \ldots, u_{im})^\T \in \mathbb{P}^{m \times 1}$ and $v_i=(v_{i1}, \ldots, v_{in})^\T \in \mathbb{P}^{n \times 1}$ such that 
$u_{ik}\cdot u_{jk}= v_{ik}\cdot v_{jk}=0$ for all $i,j,k$, $i \neq j$.
\end{thm}

\begin{proof}
    Let $A$ be a generalized well-settled matrix. 
    Then $A$ can be expressed in the form given in equation \eqref{gws eq}. 
    We will use the well-known fact that a rank-one matrix can be represented as an outer product.
    Therefore, by expressing each rank-one Bohemian block $A_i$'s as an outer product, we obtain $A_i=\zeta_{i} {\eta_{i}}^\T$, for $\zeta_i \in \mathbb{P}^{m_i\times 1}$ and $\eta_i \in \mathbb{P}^{n_i\times 1}$. This leads to the conclusion
    \begin{equation} \label{pf gws}
    \begin{split}
     PAQ &= \begin{pmatrix}
   A_1 & \hdots & 0\\
        \vdots & \ddots & \vdots\\
        0 & \hdots& A_r \end{pmatrix}=\begin{pmatrix}
   \zeta_{1} {\eta_{1}}^\T & \hdots & 0\\
        \vdots & \ddots & \vdots\\
        0 & \hdots& \zeta_{r} {\eta_{r}}^\T \end{pmatrix} \\
        &= \begin{pmatrix}
            \zeta_{1} \\ 0
        \end{pmatrix} \begin{pmatrix}
            {\eta_{1}}^\T & 0
        \end{pmatrix} + \begin{pmatrix}
            0 \\ \zeta_{2} \\ 0
        \end{pmatrix} \begin{pmatrix}
            0 & {\eta_{2}}^\T & 0
        \end{pmatrix}+ \cdots + \begin{pmatrix}
           0\\ \zeta_{r} 
        \end{pmatrix} \begin{pmatrix}
           0 & {\eta_{r}}^\T
        \end{pmatrix} .
    \end{split}
\end{equation}
 Choose $u_i=\begin{pmatrix}
            u_{i1} \\ \vdots \\ u_{im}
        \end{pmatrix}=\begin{pmatrix}
            0 \\ \zeta_{i} \\ 0
        \end{pmatrix}$ and $v_i^{\T}=\begin{pmatrix}
            v_{i1} & \cdots & v_{in}
        \end{pmatrix}=\begin{pmatrix}
            0 & {\eta_{i}}^\T & 0
        \end{pmatrix}$ for $i \in \{1, \ldots, r\}$.
    Therefore, $u_{ik}\cdot u_{jk}=0$ and $v_{ik}\cdot v_{jk}=0$ for all $i,j,k$, $i \neq j$. 
On the other hand, let $ A=\displaystyle \sum_{i=1}^{r} {u_iv_i^\T} $ such that $u_{ik}\cdot u_{jk}=0$ and $v_{ik}\cdot v_{jk}=0$ for all $i,j,k$, $i \neq j$. 
    Then, with the help of suitable permutation matrices, $PAQ$ can be written in the form of \eqref{pf gws}. 
\end{proof}

Therefore, up to permutation equivalence, generalized well-settled matrices can be characterized as block diagonal matrices, where each diagonal block is a rank-one Bohemian. 
Following the approach used for rank-one matrices in Theorem \ref{th 2.1 new}, the following result in Theorem \ref{Thm314Boh} provides an additional characterization of generalized well-settled matrices.

\begin{thm} \label{Thm314Boh}
Let $\mathbb{P}=\{0,\pm 1\}$. Any generalized well-settled matrix $A \in \mathbb{P}^{m \times n}_r$ (having the form as in equation \eqref{gws eq}) can be represented as $A=P'CA'DQ'=UA'V$, where $P', Q'$ are appropriate permutation matrices, $C \in \mathbb{P}^{m \times m}$ and $D \in \mathbb{P}^{n \times n}$ are invertible diagonal Bohemian matrices, $U=P'C, V=DQ'$ are unitary Bohemian matrices and $A'=\text{diag}(A_1', \cdots, A_r')$ is a block diagonal matrix. 
In the case $A_i$ has no zero rows, $A_i' \in \mathbb{P}^{m_i \times n_i}$ is a full matrix of type I or type III; otherwise if $A_i$ has $k$ zero rows then
 $A_i'=\begin{pmatrix}
        B_i \\0
    \end{pmatrix}$, where $B_i \in \mathbb{P}^{(m_i-k) \times n_i}$ is a full matrix of type I or type III, for all $i \in \{1,\ldots ,r\}$.
\end{thm}

\begin{proof}
    Let $A$ be of the form as in equation \eqref{gws eq}. Then, since $A_i$'s have rank one, by Theorem \ref{th 2.1 new}, $A_i$'s can be represented as $A_i=P_iC_iA_i'D_iQ_i$, where $P_i, Q_i$ are appropriate permutation matrices, $C_i \in \mathbb{P}^{m_i \times m_i}$ and $D_i \in \mathbb{P}^{n_i \times n_i}$ for $i \in \{1, \ldots , r\}$ are invertible diagonal matrices. 
    The matrix $A$ can be rewritten as 
    \begin{equation}
    \begin{split}
        P A Q &= \begin{pmatrix}
   A_1 & \hdots & 0\\
        \vdots & \ddots & \vdots\\
        0 & \hdots& A_r \end{pmatrix} = \begin{pmatrix}
   P_1C_1 & \hdots & 0\\
        \vdots & \ddots & \vdots\\
        0 & \hdots& P_rC_r \end{pmatrix}\begin{pmatrix}
   A_1' & \hdots & 0\\
        \vdots & \ddots & \vdots\\
        0 & \hdots& A_r' \end{pmatrix} \begin{pmatrix}
   D_1Q_1 & \hdots & 0\\
        \vdots & \ddots & \vdots\\
        0 & \hdots& D_rQ_r \end{pmatrix} \\
        &= \begin{pmatrix}
   P_1 & \hdots & 0\\
        \vdots & \ddots & \vdots\\
        0 & \hdots& P_r \end{pmatrix} \begin{pmatrix}
   C_1 & \hdots & 0\\
        \vdots & \ddots & \vdots\\
        0 & \hdots& C_r \end{pmatrix} \begin{pmatrix}
   A_1' & \hdots & 0\\
        \vdots & \ddots & \vdots\\
        0 & \hdots& A_r' \end{pmatrix} \begin{pmatrix}
   D_1 & \hdots & 0\\
        \vdots & \ddots & \vdots\\
        0 & \hdots& D_r \end{pmatrix} \begin{pmatrix}
   Q_1 & \hdots & 0\\
        \vdots & \ddots & \vdots\\
        0 & \hdots& Q_r \end{pmatrix}.
        \end{split}
    \end{equation}
    Hence, $A=P'CA'DQ'=UA'V$.
\end{proof}

We conclude that, similar to well-settled matrices, up to permutation equivalence, it is sufficient to determine the inner Bohemian inverses of a pure well-settled matrix to find the inner Bohemian inverses of the generalized well-settled matrices.
The latter class encompasses a broader range of matrices than well-settled matrices. 
This conclusion highlights an important relationship between the cardinalities of well-settled and generalized well-settled matrices. 
For simplicity and without loss of generality, we will assume that the permutation matrices are the identity matrices.

\begin{thm} \label{gws card}
    Let matrices $A,B,C\in \mathbb{P}^{m \times n}$ be defined as in Theorem \ref{car 3.9}, and let $D \in \mathbb{P}^{m \times n}$ be a generalized well-settled matrix with no zero rows and no zero columns,
    \begin{equation*}
        D=\begin{pmatrix}
        D_1 & \hdots & 0\\
        \vdots & \ddots & \vdots\\
        0 & \hdots& D_s \end{pmatrix}
    \end{equation*}
    such that $D_i \in \mathbb{P}^{m_i \times n_i}_1$ for all $i$. 
Then 
$$\#(A_\mathbb{P}\{1\})=\#(B_\mathbb{P}\{1\})=\#(C_\mathbb{P}\{1\}) = \#(D_\mathbb{P}\{1\}).$$
\end{thm}

\section{Inner Bohemian inverses of an extended class of matrices} \label{sec 4}

Any matrix $A\in \mathbb{K}^{m \times n}$ that belongs to the class of generalized well-settled matrices can be expressed in the form $A=\sum_{i=1}^{r} {u_iv_i^\T}$ with specific restrictions on the vectors $u_i$ and $v_i$ as outlined by Theorem \ref{gws ch1} up to permutation equivalence. 
Motivated by this characterization, we aim to define an extended class of Bohemian matrices. 
It is well-known that any matrix $A \in \mathbb{K}_r^{m \times n}$ can be written in the pattern
\begin{equation} \label{eq sum}
    A=\sum_{i=1}^{r} {u_iv_i^\T}, 
\end{equation}
where $u_i=(u_{i1}, \ldots, u_{im})^\T \in \mathbb{K}^{m \times 1}$ and $v_i=(v_{i1}, \ldots, v_{in})^\T \in \mathbb{K}^{n \times 1}$.
The set of vectors $\{u_1,\ldots,u_r\}$ and $\{v_1,\ldots,v_r\}$ are linearly independent. 
In particular, even if $A$ is Bohemian, the terms $u_iv_i^\T$ may not be Bohemian matrices. 
We are interested in Bohemian matrix $A \in \mathbb{P}_r^{m \times n}$ such that each term in the right-hand side of the equation \eqref{eq sum} is also Bohemian.
Consequently, the vectors $u_i$ and $v_i$ are restricted to being Bohemian vectors.
Specifically, it follows
\begin{equation} \label{bohemian eq sum}
    A=\sum_{i=1}^{r} {u_iv_i^\T},\ \  u_i=(u_{i1}, \ldots, u_{im})^\T \in \mathbb{P}^{m \times 1},\ v_i=(v_{i1}, \ldots, v_{in})^\T  \in \mathbb{P}^{n \times 1},
\end{equation}
with linearly independent vectors $\{u_1,\ldots,u_r\}$ and $\{v_1,\ldots,v_r\}$.

Let $A \in \mathbb{P}^{m \times n}$ be expressed in the form of equation \eqref{bohemian eq sum}.
The following two classes are considered:
\begin{enumerate}
    \item[] \textbf{Class I}: Restrict $u_i$ and $v_i$ such that $u_{ik}\cdot u_{jk}=0$ and $v_{ik}\cdot v_{jk}=0$ for all $i,j,k$, $i \neq j$.
    \item[] \textbf{Class II}: Restrict $u_i$ and $v_i$ such that $u_{ik}\cdot u_{jk}=0$ for all $i,j,k$ or $v_{ik}\cdot v_{jk}=0$ for all $i,j,k$, $i \neq j$.
  
\end{enumerate}

\begin{rem}
 Under permutation equivalence, the matrices classified as Class I take a specific form as in equation \eqref{gws eq}.
 According to Theorem \ref{gws ch1}, these matrices are generalized well-settled matrices. 
 In contrast, Class II is larger and includes both well-settled and generalized well-settled matrices. 
\end{rem}

\begin{ex}
 The matrix   $A=\begin{pmatrix}
        1 & 1& 1 \\
        1 & -1& -1 \\
        1 & -1& -1 
    \end{pmatrix}=\begin{pmatrix}
        1 \\
        0 \\
        0
    \end{pmatrix} \begin{pmatrix}
        1 & 1 & 1
        \end{pmatrix} + \begin{pmatrix}
        0 \\
        1 \\
        1
    \end{pmatrix} \begin{pmatrix}
        1 & -1 & -1
        \end{pmatrix}$
 is not a generalized well-settled matrix; instead, it belongs to the extended class of matrices, known as Class II.
\end{ex}

Let a matrix $A \in \mathbb{P}^{m \times n}$ of rank $r$ be expressed as in \eqref{bohemian eq sum}. 
In this section, we focus on matrices that fall outside the class of generalized well-settled matrices, which we refer to as Class II matrices. 
Matrices in this class can be classified into two types.
\begin{enumerate}
    \item Consider the case where $u_{ik}\cdot u_{jk}=0$ holds for all $i,j,k$, $i \neq j$. 
Up to permutation equivalence, matrices of this type will be structured in the form
\begin{equation*}
    A=\begin{pmatrix}
        A_1 \\ \vdots \\ A_r
    \end{pmatrix},
\end{equation*} 
where each block $A_i$ is a rank-one Bohemian matrix.
\item Consider the case where $v_{ik}\cdot v_{jk}=0$ for all $i,j,k$, $i \neq j$. 
Up to permutation equivalence, matrices of this type can be expressed in the form
\begin{equation*}
    A=(A_1\ |\ A_2 \ |\  \cdots \ |\  A_r),
\end{equation*} 
where each block $A_i$ is a rank-one Bohemian matrix.
\end{enumerate}

Without loss of generality, we will focus on the case $u_{ik}\cdot u_{jk}=0$ for all $i,j,k$, $i \neq j$ for the remainder of this paper. 
The results are similar in the other case.

In the subsequent result, we present a representation of Class II matrices up to permutation equivalence.
This representation is crucial for determining their generalized inverses. 
Specifically, it provides a decomposition of such a Bohemian matrix as a product of two Bohemian matrices.

\begin{thm} \label{rep case 2 1}
    Let $\mathbb{P}=\{0, \pm 1\}$. 
Consider the partitioned matrix $A=\begin{pmatrix}     A_1 \\ \vdots \\ A_r     \end{pmatrix} \in \mathbb{P}^{m \times n}$ 
with blocks satisfying $A_i \in \mathbb{P}^{m_i \times n}$. 
This matrix $A$ can be decomposed into the form $UW$, where $U \in \mathbb{P}^{m \times m}$ is a unitary matrix, and $W=\begin{pmatrix}     W_1 \\ \vdots \\ W_r    \end{pmatrix} \in \mathbb{P}^{m \times n}$,
where each $W_i$ being an $m_i \times n$ block consisting of rows that are all the same vector $w_i^\T$.
\end{thm}

\begin{proof}
According to Theorem \ref{th 2.1}, each block $A_i$ satisfies $A_i=U_iA_i'V_i$ where $U_i, V_i$ are unitary Bohemian matrices and $A_i'$ is defined as stated in Theorem \ref{th 2.1}.
    Hence, 
\begin{equation*}
    A=\begin{pmatrix}
        A_1 \\ A_2 \\ \vdots \\ A_r
    \end{pmatrix}= \begin{pmatrix}
        U_1A_1'V_1 \\ U_2A_2'V_2 \\ \vdots \\ U_rA_r'V_r
    \end{pmatrix}= \begin{pmatrix}
        U_1 & 0 & 0\\
        0 & \ddots & 0\\
         0 & 0& U_r \end{pmatrix}  \begin{pmatrix}
        A_1'V_1 \\ A_2'V_2 \\ \vdots \\ A_r'V_r
    \end{pmatrix}.
\end{equation*}
Therefore, based on the construction of $A_i'V_i$ in the proof of Theorem \ref{th 2.1}, for all $i$, there exists a vector $w_i^\T$ such that $A_i'V_i=W_i$, where $W_i$ consists of rows that are identical to $w_i^\T$.
This result concludes the proof.
\end{proof}

According to Theorem \ref{rep case 2 1}, matrices belonging to Class II can be expressed in the form \(PAQ = UW\), where \(P\) and \(Q\) are appropriate permutation matrices.
The matrices \(U\) and \(W\) are as defined in Theorem \ref{rep case 2 1}. 
Matrices in Class II, for which the row vectors $w_i^T$ of $W$ are orthogonal, as described in Theorem \ref{rep case 2 1}, are referred to as Class III matrices.

Hence, Class III matrices are defined as follows:
\begin{enumerate}
    \item[] \textbf{Class III}: Matrices of the form $PAQ=\begin{pmatrix}
        A_1 \\ \vdots \\ A_r
    \end{pmatrix} \in \mathbb{P}^{m \times n}$, $A_i \in \mathbb{P}_1^{m_i \times n}$ such that $A_iA_j^*=0$ for all $i,j, i \neq j$ and $P,Q$ are suitable permutation matrices.
\end{enumerate}

It is clear to observe the inclusions
$$Class \ I \subset Class \ III \subset Class \ II.$$

In light of Lemma \ref{1.1}, we will focus on Class III matrices, considering them up to permutation equivalence. 
Initially, let us examine the specific case of rank-two Class III matrices, represented as $A=\begin{pmatrix} A_1 \\ A_2  \end{pmatrix} \in \mathbb{P}^{m \times n}$, where $\ra{A_1}=\ra{A_2}=1$ and $A_1A_2^*$$=0$. 
According to Theorem \ref{rep case 2 1}, it follows $A=UW$.
Furthermore, using Lemma \ref{1.1}(2), we can assume $U=I_m$ without loss of generality. 
In considering all possible scenarios for rank-two Class III matrices, we conclude it is sufficient to focus on four specific structures:
    \begin{enumerate}
        \item[{\bf S1.}] $A_1=( 1_{m_1n})$, $A_2=(1_{m_2n_1}\ |\, -1_{m_2n_2})$ such that  $n_1=n_2$.
        \item[{\bf S2.}] $A_1=( 1_{m_1n})$, $A_2=(1_{m_2n_1}\ |, -1_{m_2n_2} |\ 0_{m_2n_3})$ such that  $n_1=n_2$.
        \item[{\bf S3.}] $A_1=( 1_{m_1n_1}\ |\ 1_{m_1n_2}\ | 1_{m_1n_3} | 0_{m_1n_4})$, $A_2=(1_{m_2n_1} | -1_{m_2n_2}\ |\ 0_{m_2n_3}\ |\ 1_{m_2n_4} )$.
        \item[{\bf S4.}] $A_1=( 1_{m_1n_1}\ |\ 0_{m_1n_2})$, $A_2=( 0_{m_2n_1} |\ 1_{m_2n_2} )$.
    \end{enumerate}
Among the four structures defined above, the fourth one is a well-settled matrix.
Additionally, the inner Bohemian inverses of well-settled matrices are extensively studied in \cite{chan2022inner}. 
Therefore, it suffices to focus on the first three structures.

Lemma \ref{lem2} presents an auxiliary result.

\begin{lem} \label{lem2}
       Let $A=\begin{pmatrix}
        A_1  \\ \vdots \\ A_r
    \end{pmatrix} \in \mathbb{K}^{m \times n}$ where $A_i \in \mathbb{K}^{m_i \times n}$ for all $i \in \{1,\ldots,r\}$ and ${A_i}{A_j}^*=0 $ for all $ i \neq j$, then
\begin{equation} \label{aux 1-inv}
    A\{1\} \subseteq \left\{X=(X_1 \ |\  X_2 \ |\  \cdots \ |\  X_r) \in \mathbb{K}^{n \times m} :\ X_i \in A_i\{1\} \ \forall i\right\}.
\end{equation}
\end{lem}

\begin{proof}
    Let $X=(X_1 \ |\  X_2 \ |\  \cdots \ |\  X_r) \in A\{1\}$.
        Then $AXA=A$ gives $\displaystyle A_i=\sum_{j=1}^{r}A_iX_jA_j$. 
    Postmultiplying both sides by $A_i^*$ produces $A_iX_iA_iA_i^*=A_iA_i^*$, which is equivalent to $A_iX_iA_i=A_i$. 
    Hence we obtain $X_i \in A_i\{1\}$. 

\end{proof}

Lemma \ref{lem2} can be utilized to characterize the set of $\{1\}$-inverses for specific Bohemian matrices. 
In the subsequent results, Lemma \ref{lem2} will be used to provide a complete description of $\{1\}$-inverses for the structured matrices of types {\bf S1} and {\bf S2}.

\begin{thm} \label{thm51}
    Let $A=\begin{pmatrix}
        A_1 \\ A_2
    \end{pmatrix} \in \mathbb{K}^{(m_1+m_2) \times n}$, $A_1=( 1_{m_1n})$, $A_2=(1_{m_2n_1}\ |\, -1_{m_2n_2})$ such that  $n_1=n_2$. 
Then
    \begin{equation*}
        A\{1\}=\left\{X=(X_1\ |\ X_2) \in \mathbb{K}^{n \times (m_1+m_2)} :\ \Xi(X_{11})=\Xi(X_{12})=\Xi(X_{21})=-\Xi(X_{22})=\frac{1}{2}\right\},
    \end{equation*}
    where $X_i=\begin{pmatrix}
            X_{i1}  \\            X_{i2}        \end{pmatrix} \in \mathbb{K}^{(n_1+n_2) \times m_i}$ for $i \in \{1,2\}$.
\end{thm}

\begin{proof}
 It is clear that $A_1A_2^*=0$ and therefore by Lemma \ref{lem2}, $X \in A\{1\}$ implies $X_i \in A_i\{1\}$ for $i \in \{1,2\}$. 
 Hence, let $X_i \in A_i\{1\}$. Furthermore, $AXA=A$ if and only if $A_1X_2A_2=0$ and $A_2X_1A_1=0$. 
 By Lemma \ref{multiply full matrix}, it follows  that $A_1X_2A_2=0$ if and only if $\Xi(X_2)=0$ and $A_2X_1A_1=0$ if and only if $\Xi(X_{11})=\Xi(X_{12})$. 
 Moreover, based on $X_i \in A_i\{1\}$, it follows $\Xi(X_1)=1$ and $\Xi(X_{21})-\Xi(X_{22})=1$, which concludes the proof.
\end{proof}

\begin{thm} \label{th4.5}
     Let the blocks involved in the matrix $A=\begin{pmatrix}        A_1 \\ A_2    \end{pmatrix}  \in \mathbb{K}^{(m_1+m_2) \times n}$
    satisfy $A_1=( 1_{m_1n})$, $A_2=(1_{m_2n_1}\ |\, -1_{m_2n_2}\ |\ 0_{m_2n_3})$ with $n_1=n_2$. 
In this case,
    \begin{equation*}
        A\{1\}=\left\{X=(X_1 \ |\ X_2) \in \mathbb{K}^{n \times m} :\ X_i \in A_i\{1\} \ \forall i, \ \Xi(X_{11})=\Xi(X_{12}), \ \Xi(X_{2})=0\right\},
    \end{equation*}
    where $X_i=\begin{pmatrix}
            X_{i1}  \\
            X_{i2} \\
            X_{i3}
        \end{pmatrix}\in \mathbb{K}^{(n_1+n_2+n_3) \times m_i}$ for $i \in \{1,2\}$.
\end{thm}

\begin{proof}
    It is evident that $A_1A_2^*=0$ and hence by Lemma \ref{lem2}, $X \in A\{1\}$ implies $X_i \in A_i\{1\}$. 
    Hence, let $X_i \in A_i\{1\}$. 
    Now $AXA=A$ if and only if $A_1X_2A_2=0$ and $A_2X_1A_1=0$. 
    Similarly as in Theorem \ref{thm51}, by Lemma \ref{multiply full matrix}, $A_1X_2A_2=0$ if and only if $\Xi(X_2)=0$ and $A_2X_1A_1=0$ if and only if $\Xi(X_{11})=\Xi(X_{12})$. 
    Now, since $X_i \in A_i\{1\}$, it follows $\Xi(X_1)=1$ and $\Xi(X_{21})-\Xi(X_{22})=1$, which concludes the proof.
\end{proof}

\begin{rem}

The sets of the inner Bohemian inverses for the matrices in Theorem \ref{thm51} and \ref{th4.5} can be obtained by restricting the entries of inner inverses to the set $\mathbb{P}$, since $A_{\mathbb{P}}\{1\}=A\{1\} \cap \mathbb{P}^{n \times m}$.
\end{rem}

It is important to note that Lemma \ref{lem2} provides only necessary conditions for the set of inner inverses. 
For certain specific class of structured matrices, the set of inner inverses can be characterized as outlined in Theorems \ref{thm51} and \ref{th4.5}. 
However, this approach may not be effective for larger and more complex structured matrices, such as Class III matrices with rank greater than 2. 
It is challenging to obtain a comprehensive classification of structured matrices, similar to that we obtained for rank-two Class III matrices, in order to classify all Class III matrices of rank greater than 2.  
To address these challenges, we will provide a complete characterization of the set of inner inverses for a specific class of structured matrices.
We note that by restricting this class to Bohemian matrices, we can encompass all Class III matrices of arbitrary rank. 
Furthermore, limiting the corresponding inner inverses to Bohemian matrices will yield the set of all inner Bohemian inverses.

\begin{thm} \label{lem2 iff}
       Let $A=\begin{pmatrix}
        A_1 \\  \vdots \\ A_r
    \end{pmatrix} \in \mathbb{K}^{m \times n}$ and $A_i \in \mathbb{K}^{m_i \times n}$ such that $\ra{A_i}=1$ and ${A_i}{A_j}^*=0 $ for all $ i \neq j$, then
  \begin{equation*}
    A\{1\} = \left\{X=(X_1 \ |\  X_2 \ |\  \cdots \ |\  X_r) \in \mathbb{K}^{n \times m} :\ X_i \in A_i\{1\}, \ip{v_i}{e_j} =0 \ \forall i,j, i \neq j \right\},
\end{equation*}
where $A_i=u_iv_i^\T$ such that $u_i\in \mathbb{K}^{m_i \times 1}$, $v_i \in \mathbb{K}^{n \times 1}$ and $e_i :=X_iu_i \in \mathbb{K}^{n \times 1}$ for all $i=1,\ldots, r$. 

\end{thm}
\begin{proof}
    Let $X_i \in A_i\{1\}$. Then $AXA=A$ if and only if $\displaystyle A_i=\sum_{j=1}^{r}A_iX_jA_j$ for all $ i \in \{1,\dots, r\}$ if and only if $\displaystyle \sum_{j=1, j \neq i}^{r}A_iX_jA_j=0$. Furthermore, since $X_iA_i$'s are idempotent, $X_iA_i$ can be written as 
    \begin{equation*}
        X_iA_i=X_iu_iv_i^\T= e_iv_i^\T \text{ such that } \ip{e_i}{v_i}=1.
    \end{equation*}
Therefore, 
\begin{equation}
    \sum_{j=1, j \neq i}^{r}A_iX_jA_j=\sum_{j=1, j \neq i}^{r} u_iv_i^\T e_jv_j^\T=u_i \sum_{j=1, j \neq i}^{r}  \ip{v_i}{e_j}  v_j^\T.
\end{equation}
Since $u_i \neq 0$ and $\ip{v_i}{v_j} =0 $ for all $i\neq j$, it follows that $\displaystyle \sum_{j=1, j \neq i}^{r}A_iX_jA_j=0$ if and only if $\ip{v_i}{e_j} =0$ for all $i\neq j$.
\end{proof}

Theorem \ref{Thm48Boh} provides a complete characterization of $\{1\}$-inverses for a rank-two Class III matrix with structure {\bf S3}.
This characterization is based on the results presented in Theorem \ref{lem2 iff}.

\begin{thm}\label{Thm48Boh}
      Let $A=\begin{pmatrix}
        A_1 \\ A_2
    \end{pmatrix}  \in \mathbb{K}^{(m_1+m_2) \times n}$ be a block matrix such that 
    $A_1=( 1_{m_1n_1}\ |\ 1_{m_1n_2}\ |\ 1_{m_1n_3}\ |\ 0_{m_1n_4})$, $A_2=(1_{m_2n_1}\ |\, -1_{m_2n_2}\ |\ 0_{m_2n_3} \ |\ 1_{m_2n_4} )$. 
Then $X=(X_1 \ |\ X_2) \in \mathbb{K}^{n \times (m_1+m_2)}$ satisfies $X \in A\{1\}$ if and only if 
\begin{enumerate}
    \item[$(a)$] $\Xi(X_{13})=1-\Xi(X_{11})-\Xi(X_{12}),$
    \item[$(b)$] $\Xi(X_{14})=\Xi(X_{12})-\Xi(X_{11}),$
    \item[$(c)$] $-\Xi(X_{23})=\Xi(X_{21})+\Xi(X_{22}),$
    \item[$(d)$] $\Xi(X_{24})=1+\Xi(X_{22})-\Xi(X_{21}),$
\end{enumerate}
where $X_i=\begin{pmatrix}
            X_{i1}  \\
            X_{i2} \\
            X_{i3} \\
            X_{i4}
        \end{pmatrix} \in \mathbb{K}^{(n_1+n_2+n_3+n_4) \times m_i}$ for $i \in \{1,2\}$.
\end{thm}
\begin{proof}
    Suppose $A_1=u_1v_1^\T$ and $A_2=u_2v_2^\T$. 
    Then $u_1=(1_{m_11}), u_2=(1_{m_21})$, $v_1^\T=( 1_{1(n_1+n_2+n_3)}\ |\ 0_{1n_4})$ and $v_2^\T=(1_{1n_1}\ |\, -1_{1n_2}\ |\ 0_{1n_3}\ |\ 1_{1n_4} )$. 
 For $X = (X_1\ |\  X_2) \in \mathbb{K}^{n \times (m_1+m_2)}$, it follows $e_i=X_iu_i=(R_1^i\ |\ R_2^i \ |\  \cdots \ |\  R_n^i)^\T$, where $R_j^i$ represents the sum of $j$th row of $X_i$. 
Therefore, by Theorem \ref{lem2 iff}, $\ip{v_1}{e_2} =0$ gives $-\Xi(X_{23})=\Xi(X_{21})+\Xi(X_{22})$ and $\ip{v_2}{e_1} =0$ gives $\Xi(X_{14})=\Xi(X_{12})-\Xi(X_{11})$. 
The remaining conditions $(a)$ and $(d)$ are a consequence of $X_i \in A_i\{1\}$. 
    This concludes the proof.
\end{proof}

In Example \ref{Exm410Boh}, we take a rank-three matrix from Class III and find its inner inverses using Theorem \ref{lem2 iff} to demonstrate the applicability of the result for higher rank matrices.
\begin{ex}\label{Exm410Boh}
    Let the blocks in $A=\begin{pmatrix}
        A_1 \\ A_2 \\ A_3
    \end{pmatrix}  \in \mathbb{K}^{(m_1+m_2+m_3) \times n}$ satisfy 
$A_1=( 1_{m_1n_1}\ |\ 1_{m_1n_2}\ |\ 1_{m_1n_3}\ |\ 1_{m_1n_4})$, $A_2=(1_{m_2n_1}\ |\ 1_{m_2n_2}\ |\, -1_{m_2n_3}\ |\, -1_{m_2n_4} )$ and $A_3=(1_{m_3n_1}\ |\, -1_{m_3n_2}\ |\ 0_{m_3n_3}\ |\ 0_{m_3n_4})$. 
Let $X= (X_1 \ |\  X_2 \ |\  X_3) \in \mathbb{K}^{n \times (m_1+m_2+m_3)}$ and $X_i=\begin{pmatrix}
            X_{i1}  \\            X_{i2} \\            X_{i3} \\            X_{i4}
        \end{pmatrix} \in \mathbb{K}^{(n_1+n_2+n_3+n_4) \times m_i}$ for $i \in \{1,2,3\}$. 
Clearly $A_iA_j^*=0$ for all $i  \neq j$. Since $A_i$'s are rank-one, they can be represented as outer products $A_i=u_iv_i^\T$, for all $i$.
Here, 
$u_1=(1_{m_11}), u_2=(1_{m_21}), u_3=(1_{m_31})$ and $v_1=(1_{n1}), v_2=(1_{n_11}\ |\ 1_{n_21}\ |\, -1_{n_31}\ |\, -1_{n_41} ), v_3=(1_{n_11}\ |\, -1_{n_21}\ |\ 0_{n_31}\ |\ 0_{n_41})$. 
Hence $e_i=X_iu_i=(R_1^i \ |\  R_2^i \ |\  \cdots \ |\  R_n^i)^\T$,     where $R_j^i$ represents the sum of $j$th row of $X_i$. 
    
    Now, $\ip{v_1}{e_2}=0$ $\iff$ $\Xi(X_2)=0$, $\ip{v_1}{e_2}=0$ $\iff$ $\Xi(X_3)=0$, $\ip{v_2}{e_1}=0$ $\iff$ $\Xi(X_{11})+\Xi(X_{12})-\Xi(X_{13})-\Xi(X_{14})=0$, $\ip{v_2}{e_3}=0$ $\iff$ $\Xi(X_{31})+\Xi(X_{32})-\Xi(X_{33})-\Xi(X_{34})=0$, $\ip{v_3}{e_1}=0$ $\iff$ $\Xi(X_{11})=\Xi(X_{12)}$ and $\ip{v_3}{e_2}=0$ $\iff$ $\Xi(X_{21})=\Xi(X_{22)}$.
    
    Moreover $X_i \in A_i\{1\}$ gives $\Xi(X_1)=1, \Xi(X_{21})+\Xi(X_{22})-\Xi(X_{23})-\Xi(X_{24})=1$ and $\Xi(X_{31})-\Xi(X_{32})=1$. Hence we conclude $ X \in A\{1\}$ if and only if
    \begin{enumerate}
        \item $\Xi(X_{11})=\Xi(X_{12})=\frac{1}{4}$,\ \ $\Xi(X_{13})+\Xi(X_{14})= \frac{1}{2}$,
        \item  $\Xi(X_{21})=\Xi(X_{22})=\frac{1}{4}$,\ \ $\Xi(X_{23})+\Xi(X_{24})= -\frac{1}{2}$,
        \item $\Xi(X_{31})=-\Xi(X_{32})=\frac{1}{2}$,\ \ $\Xi(X_{33})+\Xi(X_{34})= 0$.
    \end{enumerate}
\end{ex}

\section{Outer Bohemian inverses}\label{sec5}

The characterizations of outer Bohemian inverses remain an area that has yet to be fully explored. 
As noted in \cite{chan2022inner}, identifying outer Bohemian inverses requires working with affine algebraic varieties rather than linear varieties, which presents a significant challenge. 
However, we have obtained the complete set of outer Bohemian inverses for rank-one Bohemian matrices.  
Consequently, we will focus on specific subsets of outer Bohemian inverses for higher-rank matrices.
In this section, we comprehensively characterize outer Bohemian inverses for rank-one Bohemian matrices and present the cardinalities of outer Bohemian matrices for full matrices. 
Additionally, the literature indicates that outer inverses with prescribed ranks have been extensively studied \cite{weibook}. 
As a result, the rank-one outer Bohemian inverses of arbitrary rank $r$ generalized well-settled matrices are obtained. 
Furthermore, representations of outer and inner Bohemian inverses with rank one and rank $r$ are provided for a specific class of matrices.

\subsection{Rank-one Bohemian matrices}

For matrices $A \in \mathbb{K}^{m \times n}$ and $X  \in \mathbb{K}^{n \times m}$ satisfying $XAX=X$, it is straightforward to observe the rank inequality $\ra{X}\leq \ra{A}$. 
Hence when considering the rank-one matrix $A$, particularly when considering full matrices, the rank of $X$ can be either 0 (which is the trivial case) or 1. 
Throughout this subsection, a rank-one matrix $A \in \mathbb{K}^{m \times n}$ will be considered as the outer product $A=\zeta \eta^\T$ where $\zeta=(\zeta_i) \in \mathbb{K}^{m \times 1}$ and $\eta =(\eta_i) \in \mathbb{K}^{n \times 1}$. 
Likewise, throughout the subsection, let $X=pq^{\T}$ denote the rank-one outer inverse of $A$, where $p=(p_i)\in \mathbb{K}^{n \times 1}$ and $q=(q_i) \in \mathbb{K}^{m \times 1}$. 
When $X\in \mathbb{P}^{n \times m}$, the vectors $p$ and $q$ can be restricted to Bohemian vectors of the appropriate order. 
The results that follow characterize outer inverses for rank-one Bohemian matrices and provide the cardinality of the set of outer Bohemian inverses with different populations $\mathbb{P}$. 
In Theorem \ref{main 1} we will present a representation for outer inverses of rank-one matrices.

\begin{thm} \label{main 1}
Let $A=\zeta \eta^\T$, where $\zeta \in \mathbb{K}^{m \times 1}$ and $\eta \in \mathbb{K}^{n \times 1}$. Then

\begin{equation*}
    A\{2\} \setminus \{0\}=\left\{X=pq^\T \in \mathbb{K}^{n \times m} : \ip{q}{\zeta} \ip{\eta}{p} =1 \right\}.
\end{equation*}
\end{thm}

\begin{proof} Consider a non-zero outer inverse of $A$ in the form $X=pq^\T$.
Then 
\[XAX=p(q^\T \zeta) (\eta^\T p)q^\T=\alpha\beta pq^\T=pq^\T.\]
Define the numbers $\displaystyle \alpha=q^\T \zeta=\sum_{i=1}^{m}{q_i\zeta_i}$ and $\displaystyle \beta=\eta^\T p=\sum_{j=1}^{n}{\eta_jp_j}$. 
Then $XAX=X$ if and only if $\alpha\beta=1$, i.e., 
$q^\T \zeta\cdot \eta^\T p=\displaystyle \sum_{i=1}^{m}{q_i\zeta_i}\sum_{j=1}^{n}{\eta_jp_j}=\ip{q}{\zeta} \ip{\eta}{p}=1$.
\end{proof}

It has been noted in Remark \ref{remark for rankone} that determining outer inverses of full matrices of type I and type III is sufficient to establish outer inverses of rank-one Bohemian matrices.
In the subsequent corollaries, the outer inverses of the full matrix of type I are obtained, and the number of outer Bohemian inverses for the full matrix is computed in relation to various populations.
\begin{cor} \label{cor5.2}
    For $A=(1_{mn}) \in \mathbb{K}^{m \times n}$, it follows
        $$A\{2\} \setminus \{0\}=\left\{X= pq^\T \in \mathbb{K}^{n \times m}:\ \sum_{i=1}^{m}{q_i}\sum_{j=1}^{n}{p_j}=1 \right\}.$$
\end{cor}

\begin{proof}
    The proof follows from Theorem \ref{main 1} since the matrix $A$ can be written as the outer product $A=(1_{mn})=(1_{m1})(1_{n1})^\T$.
\end{proof}

\begin{cor} 
    For $A=(1_{mn}) \in \mathbb{K}^{m \times n}$, let $ 1 \in \mathbb{P} \subset  \mathbb{N}$.
    \begin{enumerate}
       
        \item If $0 \in \mathbb{P}$, then $\#(A_{\mathbb{P}}\{2\})=mn+1$.
        \item If $0 \notin \mathbb{P}$, then $\#(A_{\mathbb{P}}\{2\})=\begin{cases}
        1, & \text{if } m=n=1\\
        0, & \text{otherwise}
        \end{cases}$.
    \end{enumerate}
\end{cor}

\begin{proof}
\begin{enumerate}
   
\item Given $\mathbb{P} \subset \mathbb{N}$ and assume $0 \in \mathbb{P}$. 
From Corollary \ref{cor5.2}, $X= pq^\T \in A\{2\}$ if and only if $\sum_{i=1}^{m}{q_i}\sum_{j=1}^{n}{p_j}=1$. 
    Moreover $X\in A_{\mathbb{P}}\{2\}=A{\{2\}} \cap \mathbb{P}^{n \times m}$ implies $\sum_{i=1}^{m}{q_i}\sum_{j=1}^{n}{p_j}=1$, if either $\sum_{i=1}^{m}{q_i}=1\sum_{j=1}^{n}{p_j}=1$ or $\sum_{i=1}^{m}{q_i}=\sum_{j=1}^{n}{p_j}=-1$. 
But since $-1 \notin \mathbb{P}$, only the former case remains. 
    Since, there are $m$ such $p$ vectors and $n$ such $q$ vectors and hence the result.
\item In the case $0 \notin \mathbb{P}$, from the proof of part $1$ we conclude $m=n=1$. 
Hence $A_{\mathbb{P}}\{2\} \neq \emptyset $ if and only if $mn=1$.
\end{enumerate}
\end{proof}

\begin{cor} \label{cor 5.4}
     For $A=(1_{mn}) \in \mathbb{K}^{m \times n}$, if $\mathbb{P}=\{0, \pm{1} \}$.
    $$\#(A_{\mathbb{P}}\{2\} \setminus \{0\})= \left(\sum_{s_1=0}^{\lfloor{m-1}/{2}\rfloor} {m \choose s_1}{{m-s_1} \choose {s_1+1}}\right) \left(\sum_{s_1=0}^{\lfloor{n-1}/2\rfloor} {n \choose s_1}{{n-s_1} \choose {1+s_1}}\right).$$
   
\end{cor}
\begin{proof}
For a matrix $X$, denote by $N^{+}(X)$ and $N^{-}(X)$ the number of entries of $X$ equal to 1 and equal to $-1$, respectively. 
     From Corollary \ref{cor5.2}, $X= pq^\T \in A\{2\}$ if and only if $\sum_{i=1}^{m}{q_i}\sum_{j=1}^{n}{p_j}=1$. 
Moreover $X \in A_{\mathbb{P}}\{2\}=A{\{2\}} \cap \mathbb{P}^{n \times m}$ implies $\sum_{i=1}^{m}{q_i}\sum_{j=1}^{n}{p_j}=1$, if either $\sum_{i=1}^{m}{q_i}=\sum_{j=1}^{n}{p_j}=1$ or $\sum_{i=1}^{m}{q_i}=\sum_{j=1}^{n}{p_j}=-1$. 
     Since $X=pq^\T=(-p)(-q)^\T$, both the cases are giving exactly the same matrices $X$. 
     Hence, without loss of generality, we consider the first case. 
     Let  $N^{+}(p)=s_1+1$ and  $N^{-}(p)=s_1$, where $s_1$ has choices $\{-n,-(n-1),\ldots,0,1,\ldots,n\}$ such that $s_1+s_1+1\leq n$, which gives the upper bound for $s_1$.
     
     Now, the number of ways to choose $s_1$ number of $-1$'s from $n$ entries of $p$ are ${n \choose s_1}$. 
Hence, the number of possible vectors $p$ satisfying $\sum_{j=1}^{n}{p_j}=1$ is equal to
$$S_1=\sum_{s_1=0}^{\lfloor{n-1}/2\rfloor} {n \choose s_1}{{n-s_1} \choose {1+s_1}}$$ 
and in a similar approach we get the number of possible $q$ vectors with $\sum_{i=1}^{m}{q_i}=1 $ are 
$$S_2=\sum_{s_1=0}^{\lfloor{m-1}/{2}\rfloor} {m \choose s_1}{{m-s_1} \choose {s_1+1}}.$$ 
     Hence the result.
\end{proof}

In the following results, representations of outer inverses and the cardinality of the set of outer Bohemian inverses for full matrices of type III are obtained.

\begin{thm} \label{thm 5.5}
  If $A=(1_{mn_1}\ |\ 0_{mn_2}) \in \mathbb{K}^{m \times n}$ is defined such that $n_1+n_2=n$, then
     $$A\{2\} \setminus \{0\} =\left\{X=pq^\T  \in  \mathbb{K}^{(n_1+n_2)\times m}:\ \sum_{i=1}^{m}{q_i} \sum_{j=1}^{n_1}{p_j}=1 \right\}.$$
\end{thm}

\begin{proof}
Since $A=(1_{mn_1}\ |\ 0_{mn_2})=\zeta \eta^\T$, where $\zeta=(1_{m1})$, $\eta= \begin{pmatrix}   1_{n_11} \\      0_{n_21}  \end{pmatrix}$ such that $n_1+n_2=n$, the proof follows from Theorem \ref{main 1}.
\end{proof}
\begin{cor} Let dimensions of blocks involved in $A=(1_{mn_1} |\ 0_{mn_2}) \in \mathbb{K}^{m \times n}$ satisfy $n_1+n_2=n$. 
If $\mathbb{P}=\{0, \pm{1}\}$, then the cardinal of nonzero outer Bohemian inverses of $A$ is given by
\begin{equation*}
     \#(A_{\mathbb{P}}\{2\} \setminus \{0\})=3^{n_2} \left( \sum_{s=0}^{\lfloor{m-1}/{2}\rfloor} {m \choose s}{{m-s} \choose {s+1}} \right)\left( \sum_{s=0}^{\lfloor{n_1-1}/{2}\rfloor} {n_1 \choose s}{{n_1-s} \choose {s+1}} \right).
\end{equation*}
\end{cor}

\begin{proof}
    The proof follows by Theorem \ref{thm 5.5} and is similar to the proof of Corollary \ref{cor 5.4}.
\end{proof}

All the rank-one Bohemian matrices over the population $\mathbb{P}=\{0, \pm 1\}$ possess a non-trivial outer Bohemian inverse within the same population. 
The zero matrix is regarded as the trivial outer inverse.

\begin{cor}
    For Bohemian matrices $A=(1_{mn})$ and $B=(1_{mn_1}\ |\, -1_{mn_2})$ over $\mathbb{P}=\{0,\pm 1 \}$ such that $n_1+n_2=n$, it follows that $\#(A_\mathbb{P}\{2\})=\#(B_\mathbb{P}\{2\})$.
\end{cor}

\begin{proof}
    The proof follows from Lemma \ref{1.1}, Corollary \ref{cardinaliy} and the equation \eqref{eq thm 3.2}.
\end{proof}

\begin{rem}
Similar to the case of inner inverses, the sets of Bohemian outer inverses for any two rank-one Bohemian matrices $A,B\in \mathbb{P}^{m \times n}$ over $\mathbb{P}=\{0,\pm 1 \}$ having the same number of zero rows and zero columns, are of the same cardinality: $\#(A_\mathbb{P}\{2\})=\#(B_\mathbb{P}\{2\})$.
\end{rem}

\begin{rem}
    The rank-one outer inverses of a rank-one matrix $A$ are $\{1,2\}$-inverses of $A$, also known as reflexive inner inverses.
\end{rem}

\subsection{Generalized well-settled matrices and extended class of matrices}

In Remark \ref{remark for rankone}, we noted that it is sufficient to find outer inverses of the full matrices of type I and type III in order to determine outer inverses or outer Bohemian inverses of the rank-one Bohemian matrices. 
Additionally, inner inverses of generalized well-settled matrices are discussed in Section \ref{sec 3.2}. 
Similar to these cases, let $A=\begin{pmatrix} 
        A_1 & \hdots & 0\\        \vdots & \ddots & \vdots\\     0 & \hdots& A_r \end{pmatrix} \in \mathbb{K}^{m \times n}$, 
be the matrix defined by blocks where $A_i=(1_{m_in_i})$ or $A_i=(1_{m_in_{i1}}\ |\ 0_{m_in_{i2}})$ such that $n_{i1}+n_{i2}=n_i$. 
Then, up to multiplication by permutation matrices, $A$ can be transformed into the form, $A=(B \ | \ 0)$ where $B$ is a $(1_{mn})$-pure well-settled matrix.
        
Thus, outer inverses of any well-settled matrix can be generated using outer inverses of the following matrices:
\begin{enumerate}
    \item[1)] $(\pm1_{mn})$-pure well-settled matrices,
    \item[2)] $(A \ | \ 0)$ where $A$ is a $(\pm1_{mn})$-pure well-settled matrix.
\end{enumerate}

\subsubsection{Rank-one outer Bohemian inverse}

Let $X=pq^{\T} \in \mathbb{K}^{n \times m}$ denote the rank-one outer inverse throughout this subsection, where $p=(p_i)  \in \mathbb{K}^{n \times 1}$ and $q=(q_i)  \in \mathbb{K}^{m \times 1}$ are vectors of appropriate order. 
Throughout this subsection, the partitioned form of any vector $p$ into $r$ blocks is denoted as 
$p=\begin{pmatrix}        p^1\\ \vdots \\ p^r    \end{pmatrix}$.
Theorem \ref{thm 5.10} determines rank-one outer inverses of block diagonal matrices.   

\begin{thm} \label{thm 5.10}
    Let $A=\begin{pmatrix}
        A_1 & \hdots & 0\\
        \vdots & \ddots & \vdots\\
        0 & \hdots & A_r \end{pmatrix}  \in \mathbb{K}^{m \times n}$, $A_i \in \mathbb{K}^{m_i \times n_i}$ for all $i \in \{1,\ldots ,r\}$. 
Then
\begin{equation*}
    A\{2\}_1=\left\{X=pq^\T \in \mathbb{K}^{n \times m} :\ \ \sum_{i=1}^{r} {{q^i}^\T A_i\, p^i}=1 \right\},
\end{equation*}
where $q=\begin{pmatrix}    q^1\\    \vdots\\    q^r\end{pmatrix} \in \mathbb{K}^{m \times 1}$, $p=\begin{pmatrix}    p^1\\    \vdots\\    p^r\end{pmatrix} \in \mathbb{K}^{n \times 1}$, 
$p_i \in \mathbb{K}^{n_i \times 1}$ and $q_i \in \mathbb{K}^{m_i \times 1}$.
\end{thm}

\begin{proof}
The matrix $X$ can be partitioned as follows:
    $$X=pq^\T=p({q^1}^\T \ | \ \cdots \ | \  {q^r}^\T)=\begin{pmatrix}     p^1\\ \vdots \\ p^r    \end{pmatrix}q^\T.$$ 
As a result, it can be obtained 
    \begin{equation*}
        \begin{split}
            XAX &=p({q^1}^\T \ | \ \cdots \ | \  {q^r}^\T)\begin{pmatrix}        A_1 & \hdots & 0\\        \vdots & \ddots & \vdots\\        0 & \hdots& A_r \end{pmatrix}
            \begin{pmatrix}        p_1\\ \vdots \\ p_r    \end{pmatrix}q^\T \\    
        &= \sum_{i=1}^{r} {{q^i}^\T A_i\, p^i}pq^\T.
        \end{split}
    \end{equation*}
    Therefore, $XAX=X$ if and only if $\sum_{i=1}^{r} {{q^i}^\T A_ip^i}=1$.
\end{proof}

Rank-one outer inverses of the well-settled matrices can be obtained by considering the blocks $A_i$ in Theorem \ref{thm 5.10} as full matrices. 
The resulting outcomes provide rank-one outer inverses for pure well-settled matrices.

\begin{cor} \label{rk 1 WS}
    Let $A=\begin{pmatrix}
        A_1 & \hdots & 0\\
        \vdots & \ddots & \vdots\\
        0 & \hdots& A_r \end{pmatrix} \in \mathbb{K}^{m \times n}$, such that $A_i =(1_{m_in_i})$ for all $i=1,\ldots ,r$. 
Then
\begin{equation*}
    A\{2\}_1=\left\{X=pq^\T \in \mathbb{K}^{n \times m} :\ \sum_{i=1}^{r} {\Xi(q^i)\, \Xi(p^i)}=1 \right\},
\end{equation*}
where $q=\begin{pmatrix}
    q^1\\
    \vdots\\
    q^r
\end{pmatrix} \in \mathbb{K}^{m \times 1}$, $p=\begin{pmatrix}
    p^1\\
    \vdots\\
    p^r
\end{pmatrix} \in \mathbb{K}^{n \times 1}$, $p_i \in \mathbb{K}^{n_i \times 1}$ and $q_i \in \mathbb{K}^{m_i \times 1}$.
\end{cor}

\begin{proof}
    Let $X=pq^\T$. 
    Then $XAX=X$ gives $\sum_{i=1}^{r} {{q^i}^\T A_ip^i}=1$. 
Since $A_i=u_iv_i^\T$ where $u_i=(1_{m_i1})$ and $v_i=(1_{n_i1})$, it further implies $\sum_{i=1}^{r} {{q^i}^\T u_i{v_i}^\T p^i}=1$, which concludes the proof.
\end{proof}

\begin{cor} \label{rk 1 WS2}
    Let $A=(B_{mn_{1}} | \ 0_{mn_{2}}) \in \mathbb{K}^{m \times n}$ such that 
$n_1+n_2=n$ and $B=\begin{pmatrix}
        B_1 & \hdots & 0\\
        \vdots & \ddots & \vdots\\
        0 & \hdots& B_r \end{pmatrix}$, such that $B_i =(1_{m_in_{i1}})$ for all $i \in \{1,\ldots ,r\}$. Then
\begin{equation*}
    A\{2\}_1=\left\{X=pq^\T \in \mathbb{K}^{n \times m} : \sum_{i=1}^{r} {\Xi(q^i).\Xi(\zeta ^i)}=1 \right\},
\end{equation*}
where $q=\begin{pmatrix}
    q^1\\
    \vdots\\
    q^r
\end{pmatrix} \in \mathbb{K}^{m \times 1}$, $p=\begin{pmatrix}
    \zeta \\ \eta
\end{pmatrix} \in \mathbb{K}^{(n_1+n_2) \times 1}, \zeta =\begin{pmatrix}
    \zeta^1\\
    \vdots\\
    \zeta^r
\end{pmatrix} \in \mathbb{K}^{n_1 \times 1}$, $q^i \in \mathbb{K}^{m_i \times 1}$, $\zeta^i \in \mathbb{K}^{n_{i1} \times 1}$.

\end{cor}
\begin{proof}
    Let $X=pq^\T$, where the vector $p$ is partitioned as $\begin{pmatrix}        \zeta \\ \eta    \end{pmatrix} \in \mathbb{K}^{(n_1+n_2) \times 1}$. 
    Then $X$ is partitioned in the form
     $X=\begin{pmatrix}        \zeta q^\T \\ \eta q^\T    \end{pmatrix}=\begin{pmatrix}        X_1 \\ X_2    \end{pmatrix} \in \mathbb{K}^{(n_1+n_2) \times m}$. 
Results of Corollary \ref{rk 1 WS} yield the following
    \begin{equation*}
        B\{2\}_1=\left\{X_1=\zeta q^\T \in \mathbb{K}^{n_1 \times m} :\ \sum_{i=1}^{r} {\Xi(q^i)\cdot \Xi(\zeta ^i)}=1 \right\}.
    \end{equation*}
    Also, Lemma \ref{1.2} yields
\begin{equation*}
    A\{2\}_1=\left\{X=\begin{pmatrix}            X_1 \\            X_2         \end{pmatrix} \in \mathbb{K}^{n \times m}  :\ X_1 \in B\{2\}_1, \  X_2BX_1=X_2\right\}.
\end{equation*}
Moreover, $X_2BX_1=X_2\Longleftrightarrow q^\T Bs=1\Longleftrightarrow \sum_{i=1}^r {{q^i}^\T B_i \zeta ^i}=1$, which concludes the proof.
\end{proof}

Rank-one outer inverses of the extended class of matrices are studied in the subsequent theorems. 
Theorem \ref{rank1 OI extended} gives rank-one outer inverses for a row-wise partitioned matrix.

\begin{thm} \label{rank1 OI extended}
    Let $A=\begin{pmatrix}
        A_1 \\        \vdots \\        A_r \end{pmatrix} \in \mathbb{K}^{m \times n}$, and $A_i \in \mathbb{K}^{m_i \times n}$ for all $i \in \{1,\ldots ,r\}$. 
Then
\begin{equation*}
    A\{2\}_1=\left\{X=pq^T \in \mathbb{K}^{n \times m}:\ \sum_{i=1}^{r} {{q^i}^\T A_i\, p}=1 \right\},
\end{equation*}
where $q=\begin{pmatrix}
    q^1\\   \vdots\\   q^r\end{pmatrix} \in \mathbb{K}^{m \times 1}$ and $q^i \in \mathbb{K}^{m_i \times 1}$.
\end{thm}
\begin{proof}
The proof is similar to that of Theorem \ref{thm 5.10}.

\end{proof}

Next, we represent rank-one outer inverses for full-row rank matrices. 
When these matrices are explicitly classified as Bohemian, they are categorized under Class II.
Additionally, if we restrict the set of outer inverses to be Bohemian over the same population, we obtain the outer Bohemian inverses.

\begin{thm} \label{rank1 OI extended full rank}
    Let $A=\begin{pmatrix}
        A_1 \\        \vdots \\        A_m \end{pmatrix} \in \mathbb{K}^{m \times n}_m$, and $A_i \in \mathbb{K}^{1 \times n}$  for all $i \in \{1,\ldots ,m\}$. 
    Then
\begin{equation*}
    A\{2\}_1=\left\{X=(X_1\ |\ \lambda_1X_1\ |\  \cdots \ |\  \lambda_{m-1}X_1) \in \mathbb{K}^{n \times m} :\ \sum_{i=1}^{m} {\lambda_{i-1}A_iX_1}=1, \lambda_0=1,\ \lambda_i \in \mathbb{K} \ \forall i\right\},
\end{equation*}
where $X_1 \in \mathbb{K}^{n \times 1}$.
\end{thm}

\begin{proof}
    Any rank-one matrix $X \in \mathbb{K}^{n \times m}$ can be represented as $X=(X_1\ |\ \lambda_1X_1\ |\  \cdots \ |\  \lambda_{m-1}X_1)$, where $X_1 \in \mathbb{K}^{n \times 1}$ and $\lambda_i \in \mathbb{K}$. 
It implies
        \begin{equation*}
        \begin{split}
            XAX &= (X_1\ |\ \lambda_1X_1\ |\  \cdots \ |\  \lambda_{m-1}X_1) \begin{pmatrix}         A_1 \\     \vdots \\      A_m \end{pmatrix} 
            (X_1\ |\ \lambda_1X_1\ |\  \cdots \ |\  \lambda_{m-1}X_1) \\
    &=\sum_{i=1}^{m} {\lambda_{i-1}X_1A_i} (X_1\ |\ \lambda_1X_1\ |\  \cdots \ |\  \lambda_{m-1}X_1).
        \end{split}
        \end{equation*}
        Hence $XAX=X$ if and only if $\displaystyle \sum_{i=1}^{m} {\lambda_{i-1}X_1A_i}X_1=X_1$, which is equivalent to $\displaystyle \sum_{i=1}^{m} {\lambda_{i-1}A_i}X_1=1$.
       
\end{proof}

In Section \ref{sec 4}, we examine four specific structured matrices associated with rank-two Class III matrices.
To gain insights into the outer inverses of rank-two Class III matrices, it suffices to study the outer inverses of these four specific matrices, just as we did with inner inverses. 
Corollary \ref{cor 51} represents rank-one outer inverses of these structured matrices, particularly in the context of full-rank matrices. 
It is a consequence of Theorem \ref{rank1 OI extended full rank}.

\begin{cor} \label{cor 51}
    Consider $A=\begin{pmatrix}        A_1 \\ A_2    \end{pmatrix} \in \mathbb{K}^{2 \times n}$. 
        Then
    \begin{enumerate}
        \item 
For $A_1=( 1_{1n})$, $A_2=(1_{1n_1}\ |\, -1_{1n_2})$ such that  $n_1=n_2$, it follows
    \begin{equation*}
        A\{2\}_1=\left\{X=(X_1 \ |\ \lambda X_1) \in \mathbb{K}^{n \times 2} :\ \Xi(X_{1})+\lambda ( \Xi(X_{11})-\Xi(X_{12}))=1\right\},
    \end{equation*}
    where $X_1=\begin{pmatrix}            X_{11}  \\            X_{12}        \end{pmatrix} \in \mathbb{K}^{(n_1+n_2) \times 1}$.
        \item For $A_1=( 1_{1n})$, $A_2=(1_{1n_1}\ |\, -1_{1n_2} | 0_{1n_3})$ such that  $n_1=n_2$, it follows
    \begin{equation*}
        A\{2\}_1=\left\{X=(X_1 \ |\ \lambda X_1)\in \mathbb{K}^{n \times 2} :\ \Xi(X_{1}) + \lambda (\Xi(X_{11})-\Xi(X_{12}))=1\right\},
    \end{equation*}
    where $X_1=\begin{pmatrix}
            X_{11}  \\            X_{12} \\            X_{13}        \end{pmatrix} \in \mathbb{K}^{(n_1+n_2+n_3) \times 1}$.
        \item For $A_1=( 1_{1n_1}\ |\ 1_{1n_2}\ |\ 1_{1n_3}\ |\ 0_{1n_4})$, $A_2=(1_{1n_1}\ |\, -1_{1n_2}\ |\ 0_{1n_3}\ |\ 1_{1n_4} )$, it follows
        $X=(X_1 \ |\ \lambda X_1) \in A\{2\}_1$ if and only if 
\begin{equation*}
    \Xi(X_{11})+\Xi(X_{12}) +\Xi(X_{13})+\lambda (\Xi(X_{11})-\Xi(X_{12}) +\Xi(X_{14}))=1,
\end{equation*}
where $X_1=\begin{pmatrix}
            X_{11}  \\            X_{12} \\            X_{13} \\            X_{14}
        \end{pmatrix} \in \mathbb{K}^{(n_1+n_2+n_3+n_4) \times 1}$ and $\lambda \in \mathbb{K}$.
        \item For $A_1=( 1_{1n_1}\ |\ 0_{1n_2})$, $A_2=( 0_{1n_1}\ |\ 1_{1n_2} )$, it follows
         \begin{equation*}
        A\{2\}_1=\left\{X=(X_1 \ |\ \lambda X_1) \in \mathbb{K}^{n \times 2} :\ \Xi(X_{11})+\lambda \Xi(X_{12})=1, \lambda \in \mathbb{K} \right\},
    \end{equation*}
    where $X_1=\begin{pmatrix}            X_{11}  \\            X_{12}        \end{pmatrix} \in \mathbb{K}^{(n_1+n_2) \times 1}$.
         \end{enumerate}
         
\end{cor}

It is important to note that the matrices discussed in Corollary \ref{cor 51} are Bohemian.
By restricting the set of outer inverses to Bohemian matrices concerning the same population, we obtain the outer Bohemian inverses.

\subsubsection{Rank $r$ outer Bohemian inverse}

For a matrix $A \in \mathbb{K}^{m \times n}_r$, we aim to examine outer inverses of rank $r$ (specifically $\{1,2\}$-inverses) of the Class II and Class III matrices.
This investigation will subsequently provide the rank $r$ outer inverses for well-settled matrices.
It is important to note that a result similar to Lemma \ref{lem2} does not apply to outer inverses.
For instance, Let $A=\begin{pmatrix}
    A_1 \\A_2
\end{pmatrix}$, where
 $A_1=\begin{pmatrix}
    1 & 0 & 0 \\    0 & 0 & 0
\end{pmatrix}$ and $A_2=\begin{pmatrix}
    0 & 1 & 0 \\    0 & 0 & 0\end{pmatrix} $.
 Clearly $A_1A_2^*=0$. 
Also, $X=(X_1 \ |\  X_2)=\begin{pmatrix}
    1&0&0&1\\    0&1&1&0\\    0&0&0&0\end{pmatrix} \in A\{1,2\}$. 
Here $X_1=\begin{pmatrix}
    1&0\\    0&1\\    0&0
\end{pmatrix} \in A_1 \{1\}$ and $X_2=\begin{pmatrix}
    0&1\\    1&0\\    0&0
\end{pmatrix} \in A_2 \{1\}$ but $X_1 \notin A_1 \{1,2\}$ and $X_2 \notin A_2 \{1,2\}$. 
Hence, we will study the outer inverses of full-row rank Class II matrices.

\begin{lem} \label{lemA12}
    Let $A=\begin{pmatrix}
        A_1 \\ \vdots \\ A_m
    \end{pmatrix} \in \mathbb{K}^{m \times n}_m$ and $A_i \in \mathbb{K}^{m_i \times n}_{m_i}$ for all $i$.
    Then 
\begin{equation*}
    A\{1,2\} =A\{1\} \subseteq \left\{X=(X_1\ |\ X_2 \ |\  \cdots \ |\  X_m) \in \mathbb{K}^{n \times m} :\ X_i \in A_i\{1,2\} \ \forall i\right\}.
\end{equation*}
\end{lem}

\begin{proof}
    Let $X=(X_1 \ |\ X_2 \ |\ \cdots \ |\ X_m) \in A\{1\}$. 
Since $A$ is full-row rank matrix, it follows $AX=I$, and hence $X \in A\{1,2\}$ as well as
        $$\begin{pmatrix}        A_1 \\ \vdots \\ A_m    \end{pmatrix} 
           (X_1 \ |\ X_2 \ |\ \cdots \ |\ X_m) = \begin{pmatrix}
            A_1X_1 & \cdots & A_1X_m \\            \vdots & \ddots & \vdots \\            A_mX_1 & \cdots & A_mX_m        \end{pmatrix}=I_m.$$
Therefore, $A_iX_j=0 $ for all $ i \neq j$ and $A_iX_i= I_{m_i} $ for all $ i$. 
As a result, $X_i \in A_i\{1,2\} $ for all $ i$.
\end{proof}

 \begin{thm} \label{thm 5.16}
     Let $A=\begin{pmatrix}
        A_1 \\ \vdots \\ A_m
    \end{pmatrix} \in \mathbb{K}^{m \times n}_m$ and $A_i \in \mathbb{K}^{m_i \times n}_{m_i}$ for all $i$. In this case, it follows
\begin{equation} \label{eq A12}
    A\{1,2\} =A\{1\} = \left\{X=(X_1 \ |\ X_2 \ |\ \cdots \ |\ X_m)\in \mathbb{K}^{n \times m} :\ X_i \in A_i\{1,2\}, A_iX_j=0 \ \forall i,j, i \neq j \right\}.
\end{equation}
 \end{thm}
 
 \begin{proof}
     From Lemma \ref{lemA12}, we conclude $A\{1,2\} =A\{1\}$ implies $X$ belongs to the right hand side of the equation \eqref{eq A12}. 
     Now, let $X$ belong to the right hand side of \eqref{eq A12}. 
 Since $X_i \in A_i\{1,2\}$, it follows $A_iX_i=I_{m_i}$ for all $i$. 
 Also, $A_iX_j=0$ for all $i \neq j$ gives $AX=I_m$.
Therefore both matrix equations $XAX=X$ and $AXA=A$ hold. 
 \end{proof}
 
The incidence matrix of a directed star graph is always a Bohemian matrix with respect to the population $\mathbb{P}=\{0, \pm 1\}$, and it possesses full column rank \cite{bapat2010graphs}. 
Therefore, Theorem \ref{thm 5.16} can be applied to find the set of $\{1,2\}$-inverses (or $\{1,2\}$-Bohemian inverses) of these incidence matrices.
This possibility will be demonstrated in Example \ref{ExmIncBoh}.

\begin{ex}\label{ExmIncBoh}
Consider the graph given in Figure \ref{fig:stargraph}
\begin{figure}[ht]
\centering
\begin{tikzpicture}[->, >=stealth, node distance=1cm] 

  \node[circle, draw, fill=blue!20] (1) at (0,0) {1};

  \node[circle, draw, fill=red!20] (4) at (0:1.5cm) {4};
  \node[circle, draw, fill=red!20] (2) at (90:1.5cm) {2};
  \node[circle, draw, fill=red!20] (3) at (180:1.5cm) {3};
  \node[circle, draw, fill=red!20] (5) at (270:1.5cm) {5};
\node at (0.7, -0.2) {$e_4$};
\node at (0.3, 0.7) {$e_1$};
\node at (-0.3, -0.7) {$e_3$};
\node at (-0.7, 0.2) {$e_2$};

  \draw[->] (1) -- (4);
  \draw[->] (1) -- (2);
  \draw[->] (1) -- (3);
  \draw[->] (1) -- (5);
 
\end{tikzpicture}
\caption{A directed star graph with 5 vertices}
\label{fig:stargraph}
\end{figure}
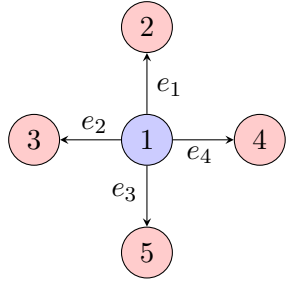

The incidence matrix of the star graph given in Figure \ref{fig:stargraph} is equal to 
$Q=\begin{pmatrix}        1 &1&1&1\\        -1&0&0&0\\        0&-1&0&0\\        0&0&0&-1\\        0&0&-1&0    \end{pmatrix}$. 
 Consider 
    $A=Q^\T=\begin{pmatrix}        1 &-1&0&0&0\\        1&0&-1&0&0\\        1&0&0&0&-1\\        1&0&0&-1&0\\    \end{pmatrix}=
    \begin{pmatrix}        A_1 \\ A_2 \\ A_3 \\ A_4    \end{pmatrix}$, where $A_i \in \mathbb{P}^{1 \times 5}$
and $X=(X_1\ |\  X_2\ |\ X_3\ |\  X_4)$, where $X_i \in \mathbb{P}^{5 \times 1}$
which satisfies the condition $X_i \in A_i\{1,2\}$ for all $i$.
It implies 
$$x_{11}-x_{12}=x_{21}-x_{23}=x_{31}-x_{35}=x_{41}-x_{44}=1.$$ 
Moreover $\ip{A_1}{X_2}=\ip{A_1}{X_3}=\ip{A_1}{X_4}=0$ gives $$x_{21}-x_{22}=x_{31}-x_{32}=x_{41}-x_{42}=0,$$
    while $\ip{A_2}{X_1}=\ip{A_2}{X_3}=\ip{A_2}{X_4}=0$ gives $$x_{11}-x_{13}=x_{31}-x_{33}=x_{41}-x_{43}=0.$$
    Further, $\ip{A_3}{X_1}=\ip{A_3}{X_2}=\ip{A_3}{X_4}=0$ gives $$x_{11}-x_{15}=x_{21}-x_{25}=x_{41}-x_{45}=0,$$ and $\ip{A_4}{X_1}=\ip{A_4}{X_2}=\ip{A_4}{X_3}=0$ gives $$x_{11}-x_{14}=x_{21}-x_{24}=x_{31}-x_{34}=0.$$ The equations obtained using the inner products give the following solution:
    \begin{equation*}
    \aligned
        x_{11}&=x_{13}=x_{14}=x_{15}=a,\ \ a\in \mathbb{K},\\
        x_{21}&=x_{22}=x_{25}=x_{24}=b,\ \ b\in \mathbb{K},\\
        x_{31}&=x_{32}=x_{33}=x_{34}=c,\ \ c \in \mathbb{K},\\
        x_{41}&=x_{42}=x_{43}=x_{45}=d,\ \ d \in \mathbb{K}.
        \endaligned
        \end{equation*}

    Hence, the set of $\{1,2\}$-inverses of $A$ is given by
    $$A\{1,2\} =\left\{X= \begin{pmatrix}
        a & b&c&d\\
        a-1 &b &c&d\\
        a &b-1&c&d\\
        a&b&c&d-1\\
        a&b &c-1&d
    \end{pmatrix} \in \mathbb{K}^{5 \times 4}, a,b,c,d \in \mathbb{K}\right\}$$ 
    and for $\mathbb{P}=\{0,\pm 1\}$, the set of reflexive inner Bohemian inverses is as follows
    $$A_{\mathbb{P}}\{1,2\} =\left\{X= \begin{pmatrix}
        a & b&c&d\\        a-1 &b &c&d\\        a &b-1&c&d\\        a&b&c&d-1\\        a&b &c-1&d    \end{pmatrix} \in \mathbb{P}^{5 \times 4}, a,b,c,d \in \{0,1\} \right\}.$$
\end{ex}

In Section \ref{sec 4}, we noted that there are four distinct cases for Class III matrices when considering rank-two matrices. 
In Theorem \ref{rank 2 outer}, we will present rank-two outer inverses for full-row rank matrices in each case.

\begin{thm} \label{rank 2 outer}
The following statements are true for the block matrix $A=\begin{pmatrix}        A_1 \\ A_2    \end{pmatrix} \in \mathbb{K}^{2 \times n}$: 
    \begin{enumerate}
        \item If $A_1=( 1_{1n})$, $A_2=(1_{1n_1}\ |\, -1_{1n_2})$ such that  $n_1=n_2$, then   
    \begin{equation*}
        A\{1,2\}=\left\{X=(X_1\ |\ X_2) \in \mathbb{K}^{n \times 2} :\ \Xi(X_{11})=\Xi(X_{12})=\Xi(X_{21})=-\Xi(X_{22})=\frac{1}{2}\right\},
    \end{equation*}
    where $X_i=\begin{pmatrix}
            X_{i1}  \\            X_{i2}
        \end{pmatrix} \in \mathbb{K}^{(n_1+n_2) \times 1} $ for $i \in \{1,2\}$.
        \item If $A_1=( 1_{1n})$, $A_2=(1_{1n_1}\ | -1_{1n_2}\ |\ 0_{1n_3})$ such that  $n_1=n_2$, then
    \begin{equation*}
        A\{1,2\}=\left\{X=(X_1\ |\ X_2) \in \mathbb{K}^{n \times 2} :\ X_i \in A_i\{1\}\, \forall i,\ \Xi(X_{11})=\Xi(X_{12}), \Xi(X_{2})=0\right\},
    \end{equation*}
    where $X_i=\begin{pmatrix}
            X_{i1}  \\            X_{i2} \\            X_{i3}
        \end{pmatrix} \in \mathbb{K}^{(n_1+n_2+n_3) \times 1}$ for $i \in \{1,2\}$.
        \item If $A_1=( 1_{1n_1}\ |\ 1_{1n_2}\ |\ 1_{1n_3} |\ 0_{1n_4})$, $A_2=(1_{1n_1}\ |\, -1_{1n_2}\ |\ 0_{1n_3}\ |\ 1_{1n_4} )$, then 
        $X=(X_1\ |\ X_2) \in \mathbb{K}^{n \times 2} \in A\{1,2\}$ 
if and only if 
\begin{enumerate}
    \item $\Xi(X_{13})=1-\Xi(X_{11})-\Xi(X_{12}),$
    \item $\Xi(X_{14})=\Xi(X_{12})-\Xi(X_{11}),$
    \item $-\Xi(X_{23})=\Xi(X_{21})+\Xi(X_{22}),$
    \item $\Xi(X_{24})=1+\Xi(X_{22})-\Xi(X_{21}),$
\end{enumerate}
where $X_i=\begin{pmatrix}
            X_{i1}  \\            X_{i2} \\            X_{i3} \\            X_{i4}        \end{pmatrix} \in \mathbb{K}^{(n_1+n_2+n_3+n_4) \times 1}$ for $i \in \{1,2\}$.
        \item If $A_1=( 1_{1n_1}\ |\ 0_{1n_2})$, $A_2=( 0_{1n_1}\ |\ 1_{1n_2} )$, then
         \begin{equation*}
        A\{1,2\}=\left\{X=(X_1\ |\ X_2) \in \mathbb{K}^{n \times 2} :\ \Xi(X_{11})= \Xi(X_{22})=1,   \Xi(X_{12})=\Xi(X_{21})=0\right\},
    \end{equation*}
    where $X_i=\begin{pmatrix}
            X_{i1}  \\            X_{i2}        \end{pmatrix} \in \mathbb{K}^{(n_1+n_2) \times 1}$ for $i \in \{1,2\}$.
        \end{enumerate}
\end{thm}

\begin{proof}
    The proof follows from the statement that $X \in A\{1\}$ and $\ra{X}=\ra{A}$ if and only if $X \in A\{1,2\}$.
\end{proof}

Therefore, according to Corollary \ref{cor 51} and Theorem \ref{rank 2 outer}, the complete set of outer inverses of a matrix $A \in \mathbb{P}^{2 \times n}_2$ contained in the Class III can be derived using
$$ A\{2\}=\{0\} \cup  A\{2\}_1 \cup A\{2\}_2.$$
In particular, a consequence of the above results is the following statement for outer inverses of well-settled matrices, verified in Theorem \ref{outer}.

\begin{thm} \label{outer}
       For $A=\begin{pmatrix}
        A_1 \\ A_2
    \end{pmatrix} \in \mathbb{K}^{2 \times n}$ defined by blocks $A_1=( 1_{1n_1}\ |\ 0_{1n_2})$, $A_2=( 0_{1n_1}\ |\ 1_{1n_2} )$, the set of outer inverses of $A$ is given by
    \begin{equation*}
        \begin{split}
            A\{2\}=\{0\} \cup  A\{2\}_1 \cup A\{2\}_2,
        \end{split}
    \end{equation*}
    where 
    \begin{equation*}
    \begin{split}
         A\{2\}_1 
         &=\left\{X=(X_1\ |\  \lambda X_1) \in \mathbb{K}^{n \times 2} :\ \Xi(X_{11})+\lambda \Xi(X_{12})=1, \lambda \in \mathbb{K}\right\},\\
        A\{2\}_2 &=\left\{X=(X_1\ |\  X_2) \in \mathbb{K}^{n \times 2} :\ \Xi(X_{11})= \Xi(X_{22})=1,   \Xi(X_{12})=\Xi(X_{21})=0\right\},
        \end{split}
    \end{equation*}
where 
$X_i=\begin{pmatrix}
            X_{i1}  \\
            X_{i2}
        \end{pmatrix} \in \mathbb{K}^{(n_1+n_2) \times 1} $ for $i \in \{1,2\}$.
\end{thm}
\begin{proof}
The proof follows from Theorem \ref{rank1 OI extended full rank} and Theorem \ref{rank 2 outer}(4).

\end{proof}

\begin{thm}
     Consider $A=\begin{pmatrix}
        A_1 \\ A_2
    \end{pmatrix}  \in \mathbb{K}^{2 \times n}$, $A_1=( 1_{1n_1}\ |\ 0_{1n_2})$, $A_2=( 0_{1n_1}\ |\ 1_{1n_2} )$. 
    Then, for $\mathbb{P}=\{0, \pm 1\}$
    \begin{equation*}
        \# A_{\mathbb{P}}\{2\}= 1+3^{n_2} \sum_{s=0}^{\lfloor{n_1-1}/{2}\rfloor} {n_1 \choose s}{{n_1-s} \choose {s+1}}+2 \sum_{s=0}^{\lfloor{n-1}/{2}\rfloor} {n \choose s}{{n-s} \choose {s+1}}+\# A_{\mathbb{P}}\{2\}_2,
    \end{equation*}
    where $$\# A_{\mathbb{P}}\{2\}_2=\sum_{s=0}^{\lfloor{n_1-1}/{2}\rfloor} {n_1 \choose s}{{n_1-s} \choose {s+1}} \sum_{s=0}^{\lfloor{n_2-1}/{2}\rfloor} {n_2 \choose s}{{n_2-s} \choose {s+1}} \sum_{s=0}^{{n_1}} {n_1 \choose s}{{n_1-s} \choose {s}}\sum_{s=0}^{{n_2}} {n_2 \choose s}{{n_2-s} \choose {s}}.$$
\end{thm}
\begin{proof}
    The set $A\{2\}$ is determined in Theorem \ref{outer}; hence its cardinality is equal to
    $$\# A_{\mathbb{P}}\{2\}= 1+ \# A_{\mathbb{P}}\{2\}_1 +\# A_{\mathbb{P}}\{2\}_2$$ 
    and
$\# A_{\mathbb{P}}\{2\}_2=\# A_{\mathbb{P}}\{1\}$. 
Consequently, the result of \cite[Corollary 3(1)]{chan2022inner} yields
    $$\# A_{\mathbb{P}}\{2\}_2=\sum_{s=0}^{\lfloor{n_1-1}/{2}\rfloor} {n_1 \choose s}{{n_1-s} \choose {s+1}} \sum_{s=0}^{\lfloor{n_2-1}/{2}\rfloor} {n_2 \choose s}{{n_2-s} \choose {s+1}} \sum_{s=0}^{{n_1}} {n_1 \choose s}{{n_1-s} \choose {s}}\sum_{s=0}^{{n_2}} {n_2 \choose s}{{n_2-s} \choose {s}}.$$
    Now we will find  
$\# A_{\mathbb{P}}\{2\}_1$, i.e., 
$$\#\left\{X=(X_1\ |\ \lambda X_1) \in \mathbb{K}^{n \times 2} :\ \Xi(X_{11})+\lambda \Xi(X_{12})=1, \lambda \in \{0,\pm 1\}\right\}.$$
 First let $\lambda=0$, then we get $\Xi(X_{11})=1$. 
 Further by \cite[Corollary 2(1)]{chan2022inner}
        $$\#\{X_{11} \in \mathbb{P}^{n_1 \times 1}:\ \Xi(X_{11})=1\}=  \sum_{s=0}^{\lfloor{n_1-1}/{2}\rfloor} {n_1 \choose s}{{n_1-s} \choose {s+1}}.$$
        As a result, 
        \begin{equation}
            \#\left\{X \in \mathbb{P}^{n \times 1}:\ \Xi(X_{11})=1\right\}=  3^{n_2} \sum_{s=0}^{\lfloor{n_1-1}/{2}\rfloor} {n_1 \choose s}{{n_1-s} \choose {s+1}}.
        \end{equation}
        Furthermore, for $\lambda=1$ the equation reduces to $\Xi(X_{11})+ \Xi(X_{12})=1$ i.e $\Xi(X_{1})=1$. 
Therefore, it is evident that
        \begin{equation}
            \#\{X \in \mathbb{P}^{n \times 1}:\ \Xi(X_{1})=1\}=  \sum_{s=0}^{\lfloor{n-1}/{2}\rfloor} {n \choose s}{{n-s} \choose {s+1}}.
        \end{equation}
Finally, let $\lambda=-1$, hence the equation reduced to $\Xi(X_{11})- \Xi(X_{12})=1$. 
Based on Theorem \ref{th 2.6} and Corollary \ref{cor 3.7}, it follows
        \begin{equation*}
            \# \{X \in \mathbb{P}^{n \times 1}:\ \Xi(X_{11})-\Xi(X_{12})=1\}= \# \{X \in \mathbb{P}^{n \times 1}:\ \Xi(X_{1})=1\}=\sum_{s=0}^{\lfloor{n-1}/{2}\rfloor} {n \choose s}{{n-s} \choose {s+1}}.
        \end{equation*}
        Therefore, $\# A_{\mathbb{P}}\{2\}_1$ is the sum of all the three cardinalities obtained:
        \begin{equation*}
            \# A_{\mathbb{P}}\{2\}_1= 3^{n_2} \sum_{s=0}^{\lfloor{n_1-1}/{2}\rfloor} {n_1 \choose s}{{n_1-s} \choose {s+1}}+2 \sum_{s=0}^{\lfloor{n-1}/{2}\rfloor} {n \choose s}{{n-s} \choose {s+1}}.
        \end{equation*}
\end{proof}
\begin{rem}
Let $A,B,C,D \in \mathbb{P}_2^{2 \times n}$ be matrices as defined in Theorem \ref{gws card} for $\mathbb{P}=\{0, \pm 1\}$.
Then 
$$\#(A_\mathbb{P}\{2\})=\#(B_\mathbb{P}\{2\})=\#(C_\mathbb{P}\{2\}) = \#(D_\mathbb{P}\{2\}).$$
\end{rem}

\begin{ex}
Let us evaluate an example from {\rm \cite{chan2022inner}}, where $A=\begin{pmatrix}    1 & 1& 0\\    1 &0 &0\end{pmatrix}=\begin{pmatrix}    A_1 \\ A_2\end{pmatrix}$, with 
$A_1=\begin{pmatrix}    1 & 1& 0\end{pmatrix}$ and $A_2= \begin{pmatrix}    1 &0 &0\end{pmatrix}$. 
Clearly, $A$ neither belongs to Class I nor to Class III. 
The matrix $A$ belongs to Class II. 
We will find the sets $A\{1\}$ and $A\{2\}$ and hence $A_{\mathbb{P}}\{1\}$ and $A_{\mathbb{P}}\{2\}$, using the previous results. 
Let $X=(X_1\ |\ X_2) \in \mathbb{K}^{3 \times 2}$, where 
$X_i=\begin{pmatrix}            x_{i1}  \\            x_{i2} \\            x_{i3}
        \end{pmatrix} \in \mathbb{K}^{3 \times 1}$ for $i \in \{1,2\}$. 
Further calculation leads to \\
$A_1X_1=x_{11}+x_{12}$, $A_2X_2=x_{21}$, $A_1X_2=x_{21}+x_{22}$ and $A_2X_1=x_{11}$. \\
Theorem \ref{thm 5.16} gives $x_{11}+x_{12}=1=x_{21}$ and $x_{21}+x_{22}=0=x_{11}$. 
Therefore,
        $$A\{1\}=A\{2\}_2=\{X=\begin{pmatrix}
            0 & 1 \\
            1 & -1\\
            x_{13} & x_{23}
        \end{pmatrix} : x_{13}, x_{23} \in \mathbb{K} \}.$$
        By Theorem \ref{rank1 OI extended full rank}, $Y=(Y_1\ |\ \lambda Y_1) \in A\{2\}_1$ if and only if $A_1Y_1+\lambda A_2Y_1=1$, i.e. $y_{11}+y_{12}+\lambda y_{11}=1$, where $\lambda \in \mathbb{K}$ and $Y_1=\begin{pmatrix}
            y_{11}  \\
            y_{12} \\
            y_{13}
        \end{pmatrix}  \in \mathbb{K}^{3 \times 1}$.
        $$A\{2\}_1=\left\{X=\begin{pmatrix}
            y_{11} & \lambda y_{11} \\
            y_{12} & \lambda y_{12}\\
            y_{13} & \lambda y_{13}
        \end{pmatrix} :\ y_{11}+y_{12}+\lambda y_{11}=1, \lambda \in \mathbb{K} \right\}.$$ 
Hence, $ A\{2\}=\{0\} \cup  A\{2\}_1 \cup A\{2\}_2$. Also, $A_{\mathbb{P}}\{1\}=A\{1\} \cap \mathbb{P}^{3 \times 2}$ and $A_{\mathbb{P}}\{2\}=A\{2\} \cap \mathbb{P}^{3 \times 2}$.
\end{ex}

\section{Concluding remarks}\label{secCon}

We conducted a systematic investigation of various classes of Bohemian matrices over the population $\mathbb{P}=\{0, \pm 1\}$. 
Our study focused on characterizing and representing both the inner and outer Bohemian inverses of these matrices. 
The classes we considered include rank-one Bohemian matrices, which encompass the full matrices \cite{chan2022inner}. 
Outer Bohemian inverses and their cardinalities are studied for rank-one Bohemian matrices. 
New insights into the inner Bohemian inverses and their cardinalities are provided. 
As a result, the open problem 2 presented in \cite{chan2022inner} is thoroughly examined for both inner and outer inverses.
Furthermore, we classified higher-rank Bohemian matrices into three classes: Class I, Class II, and Class III. 
Class I extends the category of well-settled matrices as outlined in \cite{chan2022inner}.
Class II consists of row-wise partitioned Bohemian matrices (and their transposes) that possess rank-one blocks under permutation equivalence. 
Class III matrices represent a specific subset of Class II matrices. 
Rank-two Class III matrices are categorized into four types of structured matrices. 
We examined the inner inverses of Class III matrices, along with both outer and inner inverses for full-row rank Bohemian matrices, which are a specific subclass of Class II matrices. 
We derived rank-one outer inverses as well as outer inverses of rank $r$ for full-row rank Bohemian matrices of rank $r$. 
Additionally, rank-one outer inverses are studied for Class I (generalized well-settled) matrices. 
The set of outer inverse is fully characterized for rank-two full-row rank matrices in Class III. 
In particular, we determined the cardinality for rank-two full-row rank generalized well-settled matrices. 
In conclusion, our work addresses open problem 3 concerning outer inverses of some classes of Bohemian matrices.

\smallskip
This paper also introduces several open problems for further study. 
A few of these are as follows.
\begin{enumerate}
    \item The structure of inner Bohemian inverses for matrices of Class II, excluding full-row rank matrices, remains unresolved.
    \item The structure of outer inverses for matrices in Classes I, II, and III has been partially explored in this paper, but further work is still needed.
\end{enumerate}

\medskip
\noindent {\bf Acknowledgements.} \small
Predrag Stanimirovi\'c is supported by by the Ministry of Science and Technology of China under grant H20240841\\
and by Ministry of Education, Science and Technological Development, Republic of Serbia, Contract No. 451-03-47/2023-01/200124.

\medskip
\noindent {\bf Conflict of interest}
The corresponding author states, on behalf of all authors, that there is no conflict of interest.

\bibliographystyle{elsarticle-harv}

\bibliography{ref}

\setlength{\parskip}{1pt}
  \setlength{\itemsep}{0pt plus 0.5ex}





\end{document}